\documentclass[11pt]{amsart}
\usepackage{indentfirst}
\usepackage{graphics,amsmath,amssymb,amsthm,bm}

\usepackage{amsfonts}
\usepackage{amsmath}
\usepackage{amssymb}
\usepackage{multicol}
\usepackage{mathrsfs}
\usepackage{cite}
\usepackage{epsfig}
\usepackage{color}
\usepackage{graphics}
\usepackage{graphicx}
\usepackage{epstopdf}
\usepackage{multicol,graphics}
\usepackage{hyperref}
\newcommand\bes{\begin{eqnarray}}
\newcommand\ees{\end{eqnarray}}
\newcommand\R{\mathbb R}

\newcommand{\dd}{\displaystyle}

\newcommand{\lf}{\left}
\newcommand{\rr}{\right}

\newcommand{\rd}{{\rm d }}
\newtheorem{theorem}{Theorem}[section]
\newtheorem{lemma}[theorem]{Lemma}

\newtheorem{example}[theorem]{Example}
\newtheorem{definition}[theorem]{Definition}

\newtheorem{proposition}[theorem]{Proposition}
\numberwithin{equation}{section}
\allowdisplaybreaks

\begin{document}

\title[Rate of accelerated expansion ]{Rate of accelerated expansion of the epidemic region in a nonlocal epidemic model with  free boundaries}

\author[Y. Du, W. Ni and R. Wang]{Yihong Du$^1$, Wenjie Ni$^1$ and Rong Wang$^2$}
\thanks{
$^1$ School of Science and Technology, University of New England, Armidale, NSW 2351, Australia.}
\thanks{\
$^2$ Mathematical Sciences Institute, Australian National University, Canberra, Australia.
}
\thanks{\ \ \
 Emails: ydu@une.edu.au (Y. Du), wni2@une.edu.au (W. Ni), rong.wang1@anu.edu.au (R. Wang) }

\thanks{\ \ \ This research was supported by the Australian Research Council.}

\date{\today}

\maketitle

\begin{abstract}
This paper is concerned with the long-time dynamics of an epidemic model whose diffusion
and reaction terms involve nonlocal effects described by suitable convolution operators, and the epidemic region is represented by an evolving interval enclosed by
 the free boundaries in the model. In Wang and Du \cite{WangDu-JDE}, it was shown that the model is
well-posed, and its long-time dynamical behaviour is governed by a spreading-vanishing dichotomy.
 The spreading speed was investigated in a subsequent work of Wang and Du \cite{WangDu-DCDS-A}, where a threshold condition for the diffusion kernels $J_1$ and $J_2$ was obtained,
 such that the asymptotic spreading speed is finite precisely when this  condition is satisfied. In this paper,
we examine the  case that this  threshold condition is not satisfied, which leads to accelerated spreading; for some typical classes of kernel functions,  we determine the precise rate of accelerated expansion of the epidemic region by constructing delicate upper and lower solutions.   \medskip

\smallskip


\end{abstract}

\section{introduction}

The following system has been investigated   in Wang and Du \cite{WangDu-JDE, WangDu-DCDS-A} as  a model for the spreading of cholera and other epidemics of a similar nature,
\begin{equation}\label{1}
\begin{cases}\displaystyle
u_t=d_1\int_{g(t)}^{h(t)}J_1(x-y)u(y,t)dy-d_1u-a_{11}u&\\
\displaystyle
\hspace{2.7cm}+a_{12}\int_{g(t)}^{h(t)}K(x-y)v(y,t)dy,&t>0,~~x\in(g(t),h(t)),\\
\displaystyle
v_t=d_2\int_{g(t)}^{h(t)}J_2(x-y)v(y,t)dy-d_2v-a_{22}v+G(u),&t>0,~~x\in(g(t),h(t)),\\
\displaystyle
g'(t)=\displaystyle-\mu \int_{g(t)}^{h(t)}\int_{-\infty}^{g(t)}J_1(x-y)u(x,t)dydx&\\
\displaystyle
\hspace{2.6cm}-\mu\rho \int_{g(t)}^{h(t)}\int_{-\infty}^{g(t)}J_2(x-y)v(x,t)dydx,&t>0,\\
\displaystyle
h'(t)=\mu \int_{g(t)}^{h(t)}\int^{\infty}_{h(t)}J_1(x-y)u(x,t)dydx&\\
\displaystyle
\hspace{2.6cm}+\mu\rho \int_{g(t)}^{h(t)}\int^{\infty}_{h(t)}J_2(x-y)v(x,t)dydx,&t>0,\\
u(h(t),t)=u(g(t),t)=v(h(t),t)=v(g(t),t)=0,&t>0,\\
h(0)=-g(0)=h_0,&\\
u(x,0)=u_0(x),~~~v(x,0)=v_0(x),&x\in[-h_0,h_0],
\end{cases}
\end{equation}
where the function $G(u)$ satisfies
\begin{enumerate}
\item[$(G1)$] $G\in C^{2}([0,\infty)),\ G(0)=0,\ G^\prime (z)>0\ \mbox{ for } z\geq 0;$

\item[$(G2)$] $\frac{G(z)}{z}$ is strictly decreasing and $\lim_{z\rightarrow +\infty}\frac{G(z)}{z} < \frac{a_{11}a_{22}}{a_{12}}$.
\end{enumerate}
The kernel functions  $J_1,\, J_2$ and $K$ satisfy
\begin{enumerate}
\item[$\textbf{(J):}$] $J\in C(\R)\cap L^{\infty}(\R)$, $J(x)=J(-x)\geq 0$,
 $J(0)>0$, $\displaystyle\int_{\R}J(x)dx=1\ \mbox{ for } J\in\{J_1, J_2, K\}.$
\end{enumerate}
The parameters $d_1, d_2, a_{11}, a_{12}, a_{22}, \mu, \rho, h_0$ are positive constants, and the initial functions $u_0, v_0$ satisfy
\begin{equation}\label{Assumption}
\left\{\begin{aligned}
& u_0\in C[-h_0,h_0],~~u_0(\pm h_0)=0~~\text{and}~~u_0(x)>0~~\mbox{for}~~x\in (-h_0,h_0),\\
& v_0\in C[-h_0,h_0],~~v_0(\pm h_0)=0~~\text{and}~~v_0(x)>0~~\mbox{for}~~x\in (-h_0,h_0).
\end{aligned}\right.
\end{equation}

In this model, $u(x,t)$ represents the population density of the infectious agent at spatial location $x$ and time $t$, $v(x,t)$ denotes that of the infective humans,
and $[g(t), h(t)]$ stands for the infected spatial region, which evolves with time.
The equations
\begin{equation}\label{hgprime}
\begin{cases}
g'(t)=\displaystyle-\mu \int_{g(t)}^{h(t)}\int_{-\infty}^{g(t)}J_1(x-y)u(x,t)dydx
\displaystyle
-\mu\rho \int_{g(t)}^{h(t)}\int_{-\infty}^{g(t)}J_2(x-y)v(x,t)dydx,\\
\displaystyle
h'(t)=\mu \int_{g(t)}^{h(t)}\int^{\infty}_{h(t)}J_1(x-y)u(x,t)dydx+\mu\rho \int_{g(t)}^{h(t)}\int^{\infty}_{h(t)}J_2(x-y)v(x,t)dydx
\end{cases}
\end{equation}
mean that the expanding rate of the infected region $[g(t),h(t)]$ is proportional to the outward
flux of the agents and infected human population moving out of the boundary of the
range. A detailed explanation of \eqref{hgprime} can be found in \cite{CDLL2019} (these type of free boundary conditions were also proposed independently in \cite{CQW}).

The corresponding ODE problem of
 \eqref{1} was proposed by
Capasso and Paveri-fontana \cite{Capasso1979} (see also Section 4.3.4 of \cite{C})
 to model  
oro-faecal transmitted diseases such as cholera, which has the form
\begin{equation}\label{model*ODE2}
\begin{cases}
u_t=-a_{11}u+a_{12}v,&t>0,\\
v_t=-a_{22}v+G(u),&t>0.\\
\end{cases}
\end{equation}
In this model, the  total human population is assumed to be constant during the epidemic,  $a_{11}$ stands for the natural death rate of the infectious agents,
$a_{12}$ represents the growth rate of the agents contributed by the infective humans, and $a_{22}$ denotes the fatality rate of the infective human population.
The dynamics of \eqref{model*ODE2} is governed by the following threshold parameter, known as the basic reproduction number,
\begin{equation*}
R_0:=\frac{a_{12}G'(0)}{a_{11}a_{22}}.
\end{equation*}
To be precise, if $0<R_0\leq 1$, the unique solution $(u,v)$ of \eqref{model*ODE2} with initial values $u(0),v(0)>0$ goes to $(0,0)$ as $t\to\infty$;
while if $R_0>1$, then \eqref{model*ODE2} admits a unique positive equilibrium $(u^*,v^*)$ and $(u,v)\to (u^*, v^*)$ as $t\to\infty$.
Clearly $(u^*, v^*)$ is uniquely determined by
 \begin{equation*}
 \frac{G(u^*)}{u^*}=\frac{a_{11}a_{22}}{a_{12}},\; v^*=\frac{a_{11}}{a_{12}}u^*.
 \end{equation*}

\medskip

The basic model \eqref{model*ODE2} has been refined and further developed along a number of different lines by many authors (e.g., \cite{Ahn2016, ChangDu2022, dlnz, HKLL, FZ,WangDu-DCDS-B,WLS, XLR, ZLW, zhaomeng2020, zhaomeng2020CPAA,zhaomeng2020JDE, ZW}), and we refer to the introduction section of \cite{WangDu-JDE} for a detailed account.
We note in particular that \eqref{1} is capable of describing the evolution of the epidemic region $[g(t), h(t)]$ as time changes. The dispersal of $u$ in \eqref{1} can be roughly described as following the rule that individuals of $u$ move from spatial location $x$ to $y$ with probability $J_1(x-y)$ and frequency $d_1$ in a unit of time. The dispersal of $v$ is similarly described, with $J_1(x-y)$ and $d_1$ replaced by $J_2(x-y)$ and $d_2$, respectively. The involvement of the kernel function $K$ in the equation for $u$ is caused by the modern way of handling human wastes, which  come from the infected people in a nonlocal manner via the sewage system (see \cite{WangDu-JDE} for more details).

It is proved in \cite{WangDu-JDE} that \eqref{1} has a unique solution $(u,v,g,h)$ defined for all $t>0$, and its long-time dynamics is governed by a spreading-vanishing dichotomy, namely
 one of the following cases must occur as $t\to\infty$,
\begin{itemize}
  \item [{\rm(i)}] \textbf{Vanishing}:
\begin{equation*}
\begin{cases}
  (g(t), h(t)) \mbox{ converges to a finite interval},\\
   (u(x,t),v(x,t))\to (0,0)~~\mbox{uniformly for}~~x\in[g(t),h(t)];
   \end{cases}
\end{equation*}
  \item [{\rm(ii)}]\textbf{Spreading}:
\begin{equation*}
  \begin{cases}
   (g(t), h(t))\to (-\infty, \infty),\\
   (u(x,t),v(x,t))\to (u^*,v^*) ~~\mbox{locally uniformly for $x$ in}~~\R.
   \end{cases}
\end{equation*}
\end{itemize}
Precise criteria governing the ocurrance of these alternatives have also been obtained in \cite{WangDu-JDE}. In particular, if $R_0\leq 1$, then vanishing always happens, and if $R_0>1$, then spreading happens when $h_0$ is greater than a certain threshold number $l_*$; in other words, when $R_0>1$,  the epidemic will spread successfully if the infected region is already over a certain critical size at time $t=0$, regardless of the values of $u_0$ and $v_0$ as long as they satisfy \eqref{Assumption}.

When the case of spreading happens, it is important to know how fast the epidemic spreads in space. This question was addressed in \cite{WangDu-DCDS-A}.
The answer relies on a complete understanding of the associated semi-wave problem of \eqref{1}, which  consists of the following equations in \eqref{semiwave-problem1}
and \eqref{semiwave-problem2}:
\begin{equation}\label{semiwave-problem1}
\begin{cases}
\displaystyle
d_1\int_{-\infty}^{0}J_1(x-y)\varphi_1(y)dy-d_1\varphi_1+c\varphi_1'(x)-a_{11}\varphi_1&\\
\displaystyle
\hspace{5.2cm}+a_{12}\int_{-\infty}^{0}K(x-y)\varphi_2(y)dy=0,&-\infty<x<0,\\
\displaystyle
d_2\int_{-\infty}^{0}J_2(x-y)\varphi_2(y)dy-d_2\varphi_2+c\varphi_2'(x)-a_{22}\varphi_2+G(\varphi_1)=0,&-\infty<x<0,\\
(\varphi_1(-\infty),\varphi_2(-\infty))=(u^*,v^*),~(\varphi_1(0),\varphi_2(0))=(0,0),
\end{cases}
\end{equation}
\begin{equation}\label{semiwave-problem2}
c=\mu\int^0_{-\infty}\int_0^{\infty}J_1(x-y)\varphi_1(x)dydx+
\mu\rho\int^0_{-\infty}\int_0^{\infty}J_2(x-y)\varphi_2(x)dydx.
\end{equation}
Here one seeks a solution triple $(c,\varphi_1,\varphi_2)\in (0,\infty)\times C^1((-\infty,0])\times C^1((-\infty,0])$ satisfying both \eqref{semiwave-problem1} and \eqref{semiwave-problem2}.
It was shown in \cite{WangDu-DCDS-A}  that if such a triple exists, then the spreading speed of \eqref{1} is $c$, and if such a triple does not exist, then the spreading speed of \eqref{1} is $\infty$, namely accelerated spreading happens.
Moreover, it was proved in \cite{WangDu-DCDS-A} that a solution triple to \eqref{semiwave-problem1} and \eqref{semiwave-problem2} exists if and only if the kernel functions $J_1$ and $J_2$ satisfy
\begin{enumerate}
\item[$\textbf{(J1):}$] $\displaystyle\int_0^{\infty}x J(x)dx<\infty$ for $J\in\{J_1, J_2\}$.
\end{enumerate}
More precisely,  the following two theorems were proved in \cite{WangDu-DCDS-A}.
\medskip

\noindent {\bf Theorem A} \cite{WangDu-DCDS-A}.
Suppose that $J_1, J_2, K$ satisfy {\bf (J)}, $G$ satisfies $(G1)-(G2)$, and $R_0>1$.  Then \eqref{semiwave-problem1} and \eqref{semiwave-problem2} possess a solution triple $(c,\varphi_1,\varphi_2)\in (0,\infty)\times C^1((-\infty,0])\times C^1((-\infty,0])$
with $\varphi_1$ and $\varphi_2$ non-increasing if and only if $J_1$ and $J_2$ satisfy {\bf (J1)}. Moreover, when {\bf (J1)} holds, the solution triple is unique and $\varphi_i'(x)<0$ for $x\leq 0$, $i=1,2$; such a unique solution triple is denoted by $(c_0, \varphi_1^{c_0},\varphi_2^{c_0})$.
\smallskip

\noindent {\bf Theorem B}  \cite{WangDu-DCDS-A}.
Suppose that $J_1, J_2, K$ satisfy {\bf (J)}, $G$ satisfies $(G1)-(G2)$,  and spreading happens. Then
\[
\begin{cases}
\lim_{t\to\infty}\frac{h(t)}t=-\lim_{t\to\infty}\frac{g(t)}t=c_0 & \mbox{ if  $J_1$ and $J_2$ satisfy {\bf (J1)}},\\
\lim_{t\to\infty}\frac{h(t)}t=-\lim_{t\to\infty}\frac{g(t)}t=\infty & \mbox{ otherwise},\\
\end{cases}
\]
where $c_0$ is given by Theorem A.

\medskip

The main purpose of this paper is to obtain more precise information on the expansion rate of the epidemic region $[g (t), h(t)]$ in the case $\lim_{t\to\infty}\frac{h(t)}t=-\lim_{t\to\infty}\frac{g(t)}t=\infty$, namely when the threshold condition {\bf (J1)} does not hold while spreading happens.

As expected, the rate of spreading depends crucially on the behaviour of the kernel functions near infinity.
For a constant $\alpha>0$, we define condition $(\mathbf{P_{\alpha}})$ for a kernel function $J(x)$  by
\begin{enumerate}
\item[$(\mathbf{P_{\alpha}})$:]\ \ \  $J(x)\approx |x|^{-\alpha}$  for $|x|\gg1$.
\end{enumerate}
Here, and in what follows,
\[\mbox{$\alpha(x)\approx \beta(x)$ means $C_1 \beta(x)\leq \alpha (x)\leq C_2\beta(x)$ for some positive constants $C_1$ and $C_2$,}
\]
 and for all $x$ in the range of concern.

We will determine the rate of acceleration of the spreading for \eqref{1} when the dominating one  of the two kernel functions $J_1$ and $J_2$ satisfies $(\mathbf{P_{\alpha}})$.
Clearly, for a continuous kernel function $J$ satisfying  $(\mathbf{P_{\alpha}})$,  it satisfies  {\bf (J)}  only if  $\alpha>1$, and it satisfies {\bf (J1)} only if $\alpha>2$. Thus such a kernel function $J$  satisfies {\bf (J)} but not {\bf (J1)} precisely when  $\alpha\in (1,2]$.
Consequently, if at least one of the two kernel functions $J_1$ and $J_2$ in \eqref{1} satisfies  $(\mathbf{P_{\alpha}})$ with $\alpha\in (1,2]$, then
\begin{align*}
	\int_0^{\infty}x J_1(x)dx=\infty\ \ {\rm or}\ \ \int_0^{\infty}x J_2(x)dx=\infty,
\end{align*}
and  by Theorem B, accelerated spreading will happen (whenever the spreading alternative occurs to \eqref{1}).

We say  $J_1$ is {\bf dominating} among the two functions $J_1$ and $J_2$ if  there is a constant $C> 0$ such that
\begin{align*}
	J_2(x)\leq C J_1(x) \mbox{ for } x\in\R.
\end{align*}
We define $J_2$ dominating analogously.
\medskip

Our  main result of this paper is the following theorem.
\begin{theorem}\label{th1}
Suppose that  $(\mathbf{J})$ holds,  there is a dominating kernel function among $J_1$ and $J_2$, and  this dominating kernel  satisfies $(\mathbf{P_{\alpha}})$
for some $\alpha\in (1,2]$. Then, whenever spreading happens to \eqref{1}, we have, for  $t\gg 1$,
\begin{equation*}
-g(t),h(t)\approx \begin{cases}\ t^{\frac{1}{\alpha-1}}&{\rm if}\ \alpha\in(1,2),\\
\  t\ln t&{\rm if}\ \alpha=2.
\end{cases}\end{equation*}
\end{theorem}

\noindent{\bf Remarks:}
\begin{itemize}
\item[(i)] In Theorem \ref{th1}, if  $\alpha\in (1, 2]$ is replaced by $\alpha>2$, then  problem \eqref{1}  has  finite spreading speed (by Theorem B). 
Note also that $R_0>1$ is a necessary condition for spreading to happen. Therefore, throughout this paper, we always assume that \[R_0>1.\]
\item[(ii)] We note that in Theorem \ref{th1}, the rate of acceleration is determined by the dominating kernel alone.
The non-dominating kernel  need not satisfy $(\mathbf{P_{\beta}})$ for some $\beta>0$; for example it may have compact support.
\item[(iii)] Our estimates on the acceleration rate of spreading of \eqref{1} extend those in \cite{DN2021}, where a single species model of Fisher-KPP type was considered.   To extend the estimates to a system setting with very different nonlinearities, we have to develop new techniques and ideas in this paper.
\item[(iv)] In \cite{dn2022} a rather general cooperative system of nonlocal equations with free boundaries were considered. As already pointed out in \cite{WangDu-JDE, WangDu-DCDS-A}, due to the nonlocal reaction term involving the kernel function $K(x)$ in \eqref{1}, the theory of \cite{dn2022} does not cover the model here.
\item[(v)] Let us note that the kernel function $K(\cdot)$ does not affect the rate of acceleration in Theorem \ref{th1}. However, we conjecture that in Theorem \ref{th1}, 
\[
-g(t),\ h(t)=\begin{cases} [c_\alpha+o(1)]\,t^{\frac 1{\alpha-1}} & \mbox{ if } \alpha\in (1,2),\\
[c_\alpha+o(1)]\, t\ln t & \mbox{ if } \alpha=2,
\end{cases}
\]
as $t\to\infty$ for some positive constant $c_\alpha$ which depends on $K(\cdot)$.
\item[(vi)] For the Fisher-KPP model without free boundary, rate of accelerated spreading has been obtained for some classes of diffusion kernels which are not thin tailed, see, for example, \cite{Gsiam2011}; see also \cite{CR} for the case that the nonlocal diffusion term is given by the fractional Laplacian $(-\Delta)^s$, $s\in (0,1)$.
\end{itemize}

\section{Principal sub-eigenfunctions}

In this section, we find some easily checked sufficient conditions for a function $\phi=\phi_L\in C([-L, L])$  to satisfy
\begin{align}\label{2.1a}
	\begin{cases} \displaystyle \mathcal{L}[\phi](x):=\int_{-L}^{L}J(x-y)\phi (y)\rd y\geq (1-\epsilon) \phi(x) \mbox{ for } x\in [-L, L] ,\\
	\phi(x)>0 \mbox{ for } x\in [-L,L]\ \ {\rm and}\ \ \phi(\pm L)=0,
	\end{cases}
\end{align}
 where $\epsilon>0$ is small and converges to 0 as $L\to\infty$.
Such functions (with $J\in\{J_1, J_2, K\}$) will be used as a building block in our construction of lower solutions  of \eqref{1}; they are also of independent interest.\footnote{\ A special $\phi(x)$ satisfying \eqref{2.1a} was constructed and used in \cite{DuNi2020N}}  We will call such a function $\phi$ a {\bf principal sub-eigenfunction}
of the operator $\mathcal L$ generated by the kernel function $J$ in \eqref{2.1a}.\footnote{\ Let us note that, if the kernel function $J$ satisfies condition {\bf (J)}, then the principal eigenfunction of $\mathcal L$  satisfies \eqref{2.1a} except that it is positive at $x=\pm L$. 
 Indeed, it is well known that the operator $\mathcal L_L[\phi]:=\mathcal L[\phi]-\phi$ has a principal eigenpair $(\lambda_L, \phi_L)$ and $\lambda_L\to 0$ as $L\to\infty$; see \cite{CDLL2019}. Therefore 
\[\mathcal{L}[\phi_L](x)=(1+\lambda_L)\phi_L(x)\geq (1-\epsilon)\phi_L(x) \mbox{ in } [-L, L] \mbox{ for all } L\gg 1.
\]
}

\begin{definition}\label{def2.1}  Let $\Omega$ be an open set in $\R$ and  $x_0\in  \overline\Omega$. A function $f\in C (\overline  \Omega)$  is called {\bf asymptotically convex} at $x_0$ if there is    a convex function $f_1\in C (\overline  \Omega)$ such that
	\begin{align*}
		f(x)=f_1(x)+o(f_1(x)) \mbox{ as } x\to x_0.
	\end{align*}
\end{definition}

Clearly,  $\phi\in C([-1,1])$  is asymptotically convex at $x=1$ if one of following holds:
\begin{itemize}
	\item[{\rm (a)}]  there are constants $C_1> 0$ and  $k\geq 1$ such that
	\begin{align*}
		\phi(x)=C_1(1-x)^{k}+o((1-x)^{k}) \mbox{ as } x\to 1;
	\end{align*}
\item[{\rm (b)}] for some integer $k\geq 1$ and  small $\rho>0$,  $\phi\in C^{k}([1-\rho,1])$ and
\begin{align*}
	\phi^{(i)}(1)= 0 \mbox{ for } i=0,..., k-1, (-1)^k\phi^{(k)}(1)>0,
\end{align*}
 where  $\phi^{(i)}$ denotes the $i$-{th} derivative of $\phi$.
\end{itemize}

\begin{proposition}\label{prop2.1}
	Let  $J$ satisfy {\rm \textbf{(J)}},  and $ \psi$ be Lipschitz continuous over $[-1, 1]$, positive in  $(-1,1)$, and $\psi(\pm 1)=0$. If
	\begin{align}\label{2.2a}
		\psi\ is\  asymptotically\ convex\ at \ x=\pm 1,
	\end{align}
 then for any $\epsilon>0$, there exists $r_0=r_0(\epsilon,J,\psi)>0$ such that for every $r\in (0,r_0)$,  $L:=1/r$ and $\phi(x):=\psi(rx)$, \eqref{2.1a} holds.
\end{proposition}

\begin{proof} Clearly we only need  to show
	\begin{align}\label{2.3a}
		\int_{-L}^{L}J(x-y)\phi (y)\rd y\geq (1-\epsilon) \phi(x),\ \ \ x\in [-L,L].
	\end{align}
	Since $||J||_{L^1}=1$, there is $L_0>0$ depending only on  $J$ and $\epsilon$ such that
	\begin{align}\label{2.4a}
		\int_{-L_0}^{L_0}J(x) \rd x\geq 1-\epsilon/2.
	\end{align}
	By the Lipschitz continuity of $\psi$, there is $M>0$ such that,
	\begin{align*}
		|\psi(x)-\psi(y)|\leq M|x-y|,\ \ \ |x|\leq 1,\  |y|\leq 1,
	\end{align*}
	and so
	\begin{align}\label{2.5a}
		|\phi(x)-\phi(y)|=|\psi(rx)-\psi(ry)|\leq Mr |x-y|,\ \ \ |x|\leq L,\ |y|\leq L.
	\end{align}

	{\bf Step 1}.  Extension of $\psi$
	
	By \eqref{2.2a} and Definition \ref{def2.1},   we can find a small constant $\rho>0$,   a convex function $\psi_1\in C ([1-\rho,1])$ and a function $\psi_2\in C ([1-\rho,1])$ such that
	\begin{align}\label{2.7a}
		\psi(x)=\psi_1(x)+\psi_2(x) \mbox{ for } x\in  [1-\rho,1]\ \ {\rm and}\ \  \lim_{x\nearrow 1}\frac{\psi_2(x)}{\psi_1(x)}=0,\ \  \
	\end{align}
Since $\psi(x)>0$ in $(-1,1)$, by shrinking $\rho$ if needed, we have
\begin{align*}
	\psi_1(x)> 0,  \ \ x\in [1-\rho,1),
\end{align*}
   and hence, if we extend $\psi_1$ by $\psi_1(x)=0$ for $x>1$, then
 \begin{align*}
 	\psi_1(x)\mbox{ is convex for } x\in [1-\rho,\infty).
 \end{align*}
   Similarly,  we define  $\psi(x)\equiv 0$ and $\psi_2(x)\equiv 0$ for $x>1$.

	{\bf Step 2}. Verification of \eqref{2.3a}.
	
	We will only verify \eqref{2.3a} for $x\in [0, L]$, as the case $x\in [-L, 0]$ is parallel.
	We will divide the proof into two cases: (a) $x\in \lf[0, \frac{(1-\rho_1/2)}{r}\rr]$ and  (b) $x\in \lf[\frac{(1-\rho_1/2)}{r}, \frac{1}{r}\rr]$, where $\rho_1\in (0,\rho)$  depends on $\epsilon$ but not on $r$ (see the estimate of $A_2$ in Case (b) below).
	
	{\bf Case (a)}. For
	\begin{align*}
		x\in \lf[0, \frac{(1-\rho_1/2)}{r}\rr]=\lf[0, L-\frac{\rho_1}{2r}\rr],
	\end{align*}
	a direct calculation gives
	\begin{align*}
		\int_{-L}^{L}J(x-y)\phi(y)\rd y=\int_{-L-x}^{L-x}J(y)\phi (x+y)\rd y\geq \int_{-L_0}^{L_0}J(y)\phi (x+y)\rd y,
	\end{align*}
	where $L_0$ is given by \eqref{2.4a} and we have used
	\[
	L-x=\frac 1r-x\geq \frac{1}{r}[1-(1-\rho_1/2)]=\frac{\rho_1}{2r}>L_0 \mbox{ provided that } 0<r<\frac{\rho_1}{2L_0}.
	\]
	Then by \eqref{2.4a} and \eqref{2.5a},
	\begin{align*}
		&\int_{-L_0}^{L_0}J(y)\phi (x+y)\rd y\\
		=&\ \int_{-L_0}^{L_0}J(y)\phi (x)\rd y+\int_{-L_0}^{L_0}J(y)[\phi (x+y)-\phi (x)]\rd y\\
		\geq &\ \int_{-L_0}^{L_0}J(y)\phi (x)\rd y-Mr\int_{-L_0}^{L_0}J(y)|y|\rd y\\
		\geq &\ (1-\epsilon/2)\phi(x)-MrL_0.
	\end{align*}
	Denote
	\begin{align*}
		M_1:=\min_{x\in [0,(1-\rho_1/2)/r]}\phi(x)=\min_{x\in [0,1-\rho_1/2]}\psi(x),
	\end{align*}
	which is independent of $r$.  Then from the above calculations we obtain,   for $ x\in [0, (1-\frac{\rho_1}2)/r]$,
	\begin{align*}
		&\int_{-L}^{L}J(x-y)\phi (x)\rd y\geq (1-\epsilon/2)\phi (x)-MrL_0\\
		=&\ (1-\epsilon)\phi (x)+\epsilon \phi (x)/2  -MrL_0\\
		\geq&\ (1-\epsilon)\phi (x)+ \epsilon M_1/2-MrL_0\\
		\geq &\ (1-\epsilon)\phi (x),
	\end{align*}
provided that
\begin{align}\label{2.8a}
	0<r\leq \min\left\{\frac{\epsilon M_1}{2ML_0},\frac{\rho_1}{2L_0}\right\}.
\end{align}

	{\bf Case (b)}. For
	\begin{align*}
		x\in \lf[\frac{(1-\rho_1/2)}{r}, \frac{1}{r}\rr]=\lf[L-\frac{\rho_1}{2r}, L\rr],
	\end{align*}
	we have, using $-L-x<-L_0$ and $\phi(x)=0$ for $x\geq L$,
	\begin{align*}
		\int_{-L}^{L}J(x-y)\phi (y)\rd y\geq \int_{-L_0}^{\min\{L_0, L-x\}}J(y)\phi (x+y)\rd y=\int_{-L_0}^{L_0}J(y)\phi (x+y)\rd y.
	\end{align*}
 Then it follows from \eqref{2.7a} and $\phi(s)=\psi(rs)$ that
\begin{align*}
		&\int_{-L_0}^{L_0}J(y)\phi (x+y)\rd y\\
		= & \int_{-L_0}^{L_0}J(y)\psi_1 (rx+ry)\rd y+\int_{-L_0}^{L_0}J(y)\psi_2 (rx+ry)\rd y\\
		=&\int_{0}^{L_0}J(y)[\psi_1 (rx+ry)+\psi_1 (rx-ry)]\rd y+\int_{-L_0}^{L_0}J(y)\psi_2 (rx+ry)\rd y.
\end{align*}
Since $\psi_1(s)$ is convex for $s\geq 1-\rho$, and for $x\in \lf[\frac{(1-\rho_1/2)}{r}, \frac{1}{r}\rr]$ and $y\in [0, L_0]$  we have
\[\mbox{ $rx+ry\geq rx-ry\geq 1-\rho_1/2-rL_0>1-\rho_1>1-\rho$ provided $r<\rho_1/(2L_0)$,}
\]
 we obtain, when $r<\rho_1/(2L_0)$,
\begin{align*}
	A_1:=&\int_{0}^{L_0}J(y)[\psi_1 (rx+ry)+\psi_1 (rx-ry)]\rd y\geq 	2\psi_1 (rx)\int_{0}^{L_0}J(y)\rd y\geq (1-\epsilon/2) \psi_1 (rx).
\end{align*}
Since
\begin{align*}
	\lim_{s\nearrow1}\frac{\psi_2(s)}{\psi_1(s)}=0,
\end{align*}
if $\rho_1>0$ is small enough, we have
\[
\mbox{$|\psi_2(s)|\leq \frac{\epsilon}4 \psi_1(s)$ for any $s\in [1-\rho_1/2,1]$,}
\]
  and hence
\begin{align*}
	(1-\epsilon/2) \psi_1 (s)
	=(1-\epsilon/2) \frac{\psi_1 (s)}{\psi_1 (s)+\psi_2 (s)}\psi(s)\geq \frac{1-\epsilon/2}{1+\epsilon/4} \psi(s),
\end{align*}
which, combined with $\phi(s)=\psi(rs)$, implies
\begin{align*}
	A_1\geq&(1-\epsilon/2) \psi_1 (rx)\geq \frac{1-\epsilon/2}{1+\epsilon/4} \phi(x) \mbox{ for } x\in \lf[\frac{(1-\rho_1/2)}{r}, \frac{1}{r}\rr].
\end{align*}

Note that, for $x\in \lf[\frac{(1-\rho_1/2)}{r}, \frac{1}{r}\rr]$ and $y\in [-L_0, \min\{L-x,L_0\}]$, we have
\[
rx+ry\in \left[1-\frac{\rho_1}2-rL_0, 1\right]\subset [1-\rho_1, 1] \mbox{ when \eqref{2.8a} holds.}
\]
Hence, due to   $\psi_1(rx)=\psi_2(rs)\equiv 0$ for $s>1/r$ and $|\psi_2(s)|\leq \frac{\epsilon}4 \psi_1(s)$ for $s\in [1-\rho_1/2,1]$, we obtain
\begin{align*}
	A_2:=&\left|\int_{-L_0}^{L_0}J(y)\psi_2 (rx+ry)\rd y\right|=\left|\int_{-L_0}^{\min\{L-x,L_0\}}J(y)\psi_2 (rx+ry)\rd y\right|\\
	\leq&\frac{\epsilon}{4} \int_{-L_0}^{\min\{L-x,L_0\}}J(y)\psi_1 (rx+ry)\rd y
	=\frac{\epsilon}{4}A_1.
\end{align*}

Combining  the above estimates for $A_1$ and $A_2$ yields, for $x\in \lf[\frac{(1-\rho_1/2)}{r}, \frac{1}{r}\rr]$,
\begin{align*}
	\int_{-L}^{L}J(x-y)\phi(x)\rd y\geq A_1-A_2\geq (1-\frac\epsilon 4)A_1\geq (1-\epsilon) \phi(x).
\end{align*}

The proof is finished.
\end{proof}

From the above proof, we easily obtain the following result where in \eqref{2.1a}, $L$ is replaced by $L(t)$ and  $\phi(x)$ is replaced by  $\phi(x,t)$, which is a form that
will be used in our applications.

\begin{proposition}\label{prop2.3}
	Suppose  $J$ satisfies {\rm \textbf{(J)}}, and $\psi\in C([-1,1]\times [0,\infty))$ has the following properties:
	\begin{itemize}
		\item[{\rm (1)}] $\psi(\pm 1,t)=0$ and $\psi(x,t)>0$ for $x\in (-1,1)$ and $t\in [0,\infty)$;
		\item[{\rm (2)}] $\psi(\cdot,t)$ is  Lipschitz continuous in $[-1,1]$ with Lipschitz constant $M(t)$ for $t\geq 0$;
		\item[{\rm (3)}]  there is a continuous function $\rho:[0,\infty)\to (0,1)$  such that
		\begin{align*}
			& \inf_{|x|\in [0,1-\rho(t)/2] ,\ t\in [0,\infty)}\psi(x,t)>0, \mbox{ and for every $t\geq 0$},\\
			&\psi(\cdot,t)\  \mbox{ is convex over}\  [-1,-1+\rho(t)] \mbox{ as well as over } [1-\rho(t),1].
		\end{align*}
	\end{itemize}
	Let $L\in C([0,\infty))$ be a positive  function, and define
	\begin{align*}
		\phi(x,t):=\psi\lf(\frac{x}{L(t)},t\rr),\ \  x\in [-L(t),L(t)],\ t\in [0,\infty).
	\end{align*}
 Then for any $\epsilon>0$, there exist  positive constants $r_1=r_1(\epsilon,\psi,J)$ and $r_2=r_2(\epsilon,J)$ such that if
 \begin{align}\label{2.9a}
L(t)/M(t)\geq r_1 \ \ {\rm and}\ \ L(t)\rho(t)\geq r_2  \ \ \ {\rm for\ all}\ t\geq 0,
 \end{align}
 then
	\begin{align}\label{2.10b}
		\int_{-L(t)}^{L(t)}J(x-y)\phi (y, t)\rd y\geq (1-\epsilon) \phi(x,t) \mbox{ for }  x\in [-L(t),L(t)],\ t\geq 0.
	\end{align}
\end{proposition}
\begin{proof}
This follows easily from the proof of Proposition \ref{prop2.1} by treating $t$ as a parameter. We first obtain $L_0$ as in \eqref{2.4a}, and then we use the Lipschitz continuity
to obtain the corresponding version of \eqref{2.5a}, namely, with $r(t)=1/L(t)$,
\begin{equation*}\label{2.5b}
|\phi(x, t)-\phi(y, t)|=|\psi(r(t)x, t)-\psi(r(t)y, t)|\leq M(t)r(t) |x-y| \mbox{ for } |x|\leq L(t),\ |y|\leq L(t),\ t\geq 0.
\end{equation*}

We now follow Step 1 of the previous proof with $\psi_1\equiv\psi$, $\psi_2\equiv 0$, and extend $\psi(x, t)$ to $|x|>1$ by letting $\psi(x, t)=0$ for $|x|>1$.

In Step 2, we take $\rho_1=\rho(t)$ and obtain, for Case (a),
\begin{align*}
		\int_{-L(t)}^{L(t)}J(x-y)\phi (x, t)\rd y\geq  (1-\epsilon)\phi (x, t) \mbox{ for } x\in [0, (1-\frac{\rho(t)}2)L(t)],\ t\geq 0
	\end{align*}
provided that
\begin{align}\label{2.8b}
	0<r(t)\leq \min\left\{\frac{\epsilon M_1}{2M(t)L_0},\frac{\rho(t)}{2L_0}\right\},
\end{align}
where
\begin{align*}
		M_1:= \inf_{x\in [0,1-\rho(t)/2],\, t\geq 0}\psi(x, t)>0.
	\end{align*}
	
	For Case (b), we note that $A_2=0$ and
	\[
	\int_{-L(t)}^{L(t)}J(x-y)\phi (x, t)\rd y\geq A_1\geq (1-\frac\epsilon2)\phi(x, t) \mbox{ for } x\in [(1-\frac{\rho(t)}2)L(t), L(t)],\ t\geq 0,
	\]
	provided that $0<r(t)\leq \rho(t)/(2L_0)$, in particular when \eqref{2.8b} holds. Therefore \eqref{2.10b} holds if we take
	\[
	r_1=\frac{2L_0}{M_1\epsilon} \mbox{ and } r_2=2L_0 \mbox{ in } \eqref{2.9a}.
	\]
	The proof is complete.
\end{proof}

\begin{example}\label{example2}  Let $J$ satisfy {\rm \textbf{(J)}}. 
	\begin{itemize}
		\item[{\rm (a)}] Let $\lambda\geq 1$ be a constant. Define
	\begin{align*}
		&\psi(x,t):=(1-|x|)^{\lambda},\ \ x\in [-1,1],\ t\geq 0,\\
		&\phi(x,t):=	\psi(\frac{x}{L(t)},t)=\lf(1-\frac{|x|}{L(t)}\rr)^{\lambda},\ \ x\in [-L(t),L(t)],\ t\geq 0.
	\end{align*}
 Then we can use Proposition {\rm \ref{prop2.3}} with  $M(t)\equiv \lambda$ and $\rho(t)\equiv 1/2$ to see that  \eqref{2.10b} holds if
\begin{align*}
	L(t)\geq \max\{\lambda r_1,2r_2\}.
\end{align*}
\item[{\rm (b)}] Let $\lambda\geq 2$ be a constant, and  $\eta\in C(\R^2)$ be a function whose range is contained in $[4/5, 1]$. Define
	\begin{align*}
		&\psi(x,t):= (1-|x|)^{\lambda-1} \psi_1(x,t),\ \ \ x\in [-1,1],\ t\geq 0,\\
		&\phi(x,t):=	\psi(\frac{x}{L(t)},t)=\lf(1-\frac{|x|}{L(t)}\rr)^{\lambda}\psi_1(\frac{x}{L(t)},t),\ \ x\in [-L(t),L(t)],\ t\geq 0,
	\end{align*}
where
\begin{equation*}
	\psi_1(x,t):=\begin{cases}
		1-|x|,& |x|\leq \eta(L(t), t),\\
		1-\eta(L(t), t),& \eta(L(t), t)\leq  |x|\leq 1.
	\end{cases}
\end{equation*}
Then again we can use Proposition {\rm \ref{prop2.3}} with  $M(t)\equiv \lambda$ and $\rho(t)\equiv 1/2$ to see that  \eqref{2.10b} holds if
\begin{align*}
		L(t)\geq \max\{\lambda r_1,2r_2\}.
\end{align*}
 Note that one can easily check that $\psi(\cdot, t)$ is convex over $[-1,-1/2]$ and $[1/2,1]$.
 	
	\item[{\rm (c)}]  Let $\lambda\geq 2$ be a constant, and   $\eta_i\in C(\R^2)\ (i=1,2) $ be positive functions satisfying  $\eta_1>1>\eta_2>4/5$. Define
	\begin{align*}
		&\psi(x,t):= \min\lf\{1,\big[(1-|x|)\tilde \eta_1(t)\big]^{\lambda-1} \rr\}\psi_1(x,t),\ \ \ x\in [-1,1],\ t\geq 0,
	\end{align*}
	where
	\begin{equation*}
		\psi_1(x,t):=\begin{cases}
			\min\lf\{1,(1-|x|)\tilde \eta_1(t) \rr\},& |x|\leq \tilde \eta_2(t),\\
			\min\lf\{1,[1-\tilde \eta_2(t)]\tilde \eta_1(t) \rr\},& \tilde \eta_2(t)\leq  |x|\leq 1.
		\end{cases}
	\end{equation*}
and $\tilde \eta_i(t):=\eta_i(L(t),t)$. Then  for
\begin{align*}
	\phi(x,t):=	\psi(\frac{x}{L(t)},t),\ \ x\in [-L(t),L(t)],\ t\geq 0,
\end{align*}	
there exists $r_0=r_0(\epsilon, J, \lambda)>0$ such that the inequality \eqref{2.10b}  holds provided that
\begin{align*}
	L(t)\geq  r_0\tilde \eta_1(t)  \mbox{ for all } t\geq 0.
	\end{align*}
		\end{itemize}
\end{example}

To apply Proposition {\rm \ref{prop2.3}} to (c), we note from $\lambda\geq 2$ that $\psi(x, t)$ is a convex function of $x$ when
\begin{align*}
	|x|\in \lf[1-\frac{1}{\tilde \eta_1(t)},1\rr].
\end{align*}
 So we may  take  $\rho(t):=1/\tilde \eta_1(t)$ and then
  \begin{align*}
		\psi(x,t)=\begin{cases} 1 &\mbox{ for } |x|\leq 1-\rho(t),\\
		[(1-|x|)\tilde \eta_1(t) ]^{\lambda-1}\psi_1(x,t) &\mbox{ for } 1-\rho(t)<|x|\leq 1.
		\end{cases}
\end{align*}
Therefore
 for the Lipschitz constant $M(t)$, we may take
\begin{align*}
	M(t):=  \lambda \tilde \eta_1(t).
\end{align*}
Moreover,
\[\begin{aligned}
M_1=\inf_{|x|\in [0,1-\rho(t)/2] ,\ t\in [0,\infty)}\psi(x,t)&\geq \inf_{|x|\in [0,1-\rho(t)/2] ,\ t\in [0,\infty)}[(1-|x|)\tilde \eta_1(t) ]^{\lambda-1}\min\lf\{1,(1-|x|)\tilde \eta_1(t) \rr\}\\
&\geq \inf_{t\geq 0}\, [\frac{\rho(t)}2\tilde\eta_1(t)]^{\lambda-1}\min\{1, \frac{\rho(t)}2\tilde\eta_1(t)\}=2^{-\lambda}>0.
\end{aligned}\]
Hence  \eqref{2.10b} holds if
\[
L(t)\geq r_1M(t)=r_1\lambda \tilde\eta_1(t) \mbox{ and } L(t)\geq r_2 /\rho(t)=r_2\tilde \eta_1(t),
\]
namely,
\begin{align*}
L(t)\geq \max\lf\{\lambda r_1, r_2\rr\}\tilde \eta_1(t).
\end{align*}

These examples will be used in the next section.

\section{Proof of Theorem \ref{th1}}
Theorem \ref{th1} will follow from the lemmas in this section, which is divided into three subsections.

Suppose that the assumptions of Theorem \ref{th1} hold and 
  $J_2$ is the dominating kernel function.
 Then, there are three positive constants $K_1$, $K_2$ and $K_3$ such that
\begin{equation}\label{J_1J_2}
K_1\leq \max\{1, |x|^\alpha\} J_2(x) \leq K_2 \ {\rm and}\  \frac{J_1(x)}{J_2(x)}\leq K_3 \mbox{ for }  x\in \R.
\end{equation}
Subsections 3.1 and 3.2 below give the proof of Theorem \ref{th1} when \eqref{J_1J_2} holds. The proof of Theorem \ref{th1} for the remaining case that $J_1$ is dominating will be given in Subsection 3.3.

\subsection{Lower bound for $h(t)$ and $-g(t)$}\label{subsection3.1}
We will construct a lower solution for \eqref{1}, which will provide  the desired lower bound for $h(t)$ and $-g(t)$.

We first note that there exist positive constants $\delta_1$, $\delta_2$ and $\rho$ such that for all $s\in (0,1]$,
\begin{align}\label{linear}
	-a_{11}\delta_1+a_{12}\delta_2\geq \rho (\delta_1+\delta_2)\ \mbox{ and }\
	G(s\delta_1)- a_{22}s\delta_2\geq s \rho (\delta_1+\delta_2).
\end{align}
In fact, denote
\begin{equation*}
	\mathbf{A}=\left(\begin{matrix}
		-a_{11}&a_{12}\\
		G'(0)&-a_{22}
	\end{matrix}\right).
\end{equation*}
It is easily checked that
\begin{align*}
	\rho_1=\frac{-(a_{11}+a_{22})+\sqrt{(a_{11}-a_{22})^2+4a_{12}G'(0)}}{2}
\end{align*}
is the unique positive eigenvalue of $\mathbf{A}$ due to $R_0>1$.
By the Perron-Frobenius theorem, there exists a positive  eigenvector $({\delta}_1,{\delta}_2)$ such that
\begin{align*}
	\mathbf{A}({\delta}_1,{\delta}_2)^T=\rho_1({\delta}_1,{\delta}_2)^T.
\end{align*}
By replacing $(\delta_1,\delta_2)$ with a suitable scalar multiple of it, we may assume that they are sufficiently small so that
 \begin{equation}\label{linear1}
 \begin{cases}
 &\mbox{$0<\delta_1<u^*$,\ $0<\delta_2<v^*$,
 }\\
 	&-a_{11}\delta_1+a_{12}\delta_2=\rho_1 \delta_1=\frac{\rho_1 \delta_1}{\delta_1+\delta_2}(\delta_1+\delta_2), \mbox{ and for any $s\in (0,1]$,}\\
 	&G(s\delta_1)-a_{22}s\delta_2=G'(0)s\delta_1-o(1)s{\delta_1}-a_{22}s\delta_2= \rho_1 s \delta_2-o(1) s{\delta_1}\geq \frac{s\rho_1 \delta_2}{2(\delta_1+\delta_2)}(\delta_1+\delta_2),
 \end{cases}\end{equation}
where $o(1)\rightarrow 0$ as $\delta_1\rightarrow 0$. Hence, \eqref{linear} holds with $\rho=\min\{\frac{\rho_1 \delta_1}{\delta_1+\delta_2},\frac{\rho_1 \delta_2}{2(\delta_1+\delta_2)} \}$.

\begin{lemma}\label{lemma1-2}
	Suppose that $(\mathbf{J})$ holds,  $\alpha\in (1,2)$ in \eqref{J_1J_2} and spreading occurs to \eqref{1}. Then there exists  $D:=D(\alpha)>0$ such that
	\begin{equation*}\label{hg-lower1}
	-g(t),\	h(t)\geq D\,t^{\frac{1}{\alpha-1}} \mbox{ for all large } t>0.
	\end{equation*}
\end{lemma}
\begin{proof}
	We construct a suitable lower solution to \eqref{1}, which will lead to the desired estimate by the comparison principle.	
	
Let $\lambda\geq 2$ be a constant. 	Define
	\begin{equation*}
		\begin{cases}
			\underline{h}(t):=(C_1t+\sigma_1)^{\frac{1}{\alpha-1}},~~\underline{g}(t):=-\underline{h}(t) &\mbox{ for }t\geq0,\\
		\underline{u}(x,t):=		\lf(\frac{\underline h(t)-|x|}{\underline h(t)}\rr)^{\lambda}\delta_1 &\mbox{ for } t\geq 0,\ |x|\leq \underline{h}(t),\\
	\underline{v}(x,t):=	\Psi_1(x,t)\lf(\frac{\underline h(t)-|x|}{\underline h(t)}\rr)^{\lambda-1}\delta_2 &\mbox{ for }  t\geq 0,\ |x|\leq \underline{h}(t),
\end{cases}
\end{equation*}
			where
\begin{equation*}
	\Psi_1(x,t):=
	\begin{cases}
	\dd	\frac{\underline h(t)-|x|}{\underline h(t)},&|x|\leq \psi_1(t),\\[3mm]
	\dd	\frac{\underline h(t)-\psi_1(t)}{\underline h(t)},& |x|\geq \psi_1(t),
	\end{cases}\ \ \  \ \psi_1(t):=\dd\frac{\underline h(t)}{C_2\underline h^{1-\alpha}(t)+1},
\end{equation*}
 $(\delta_1,\delta_2)$ satisfies \eqref{linear}, \eqref{linear1},  and the positive constants $C_1$, $C_2$  and $\sigma_1$ will be chosen later.

 It is clear that $\underline u$, $\underline v$ and $\underline u_t$ are continuous, while  $\underline v_t$ exists and is continuous except when  $|x|=\psi_1(t)$, where $\underline v_t$ has a jumping discontinuity. Moreover,
	\[
	0\leq \underline u (x,t)\leq \delta_1,\ 0\leq \underline v(x,t)\leq \delta_2\  \mbox{ for } |x|\leq \underline h(t),\ t\geq 0.
	\]
	In what follows, we check that $(\underline u, \underline v, \underline g, \underline h)$ defined above forms a lower solution to \eqref{1} when the involved constants
are chosen properly. 	We will do this in two steps.
	
{\bf Step 1}.	We  check the inequality
	\begin{align}\label{lower-ac1}
		\underline{h}'(t)\leq \mu \int_{-\underline{h}(t)}^{\underline{h}(t)}\int^{+\infty}_{\underline{h}(t)}\Big(J_1(x-y)\underline{u}(x,t)+\rho J_2(x-y)\underline v (x,t)\Big)\rd y \rd x,
	\end{align}
which immediately gives
\begin{align*}
	\underline{g}'(t)\geq-\mu \int_{-\underline{h}(t)}^{\underline{h}(t)}\int_{-\infty}^{-\underline{h}(t)}\Big(J_1(x-y)\underline{u}(x,t)+\rho J_2(x-y)\underline v (x,t)\Big)\rd y \rd x.
\end{align*}
	
Using the definitions of $\underline u$ and $\underline v$ we have
		\begin{align*}
			&\mu \int_{-\underline{h}(t)}^{\underline{h}(t)}\int^{+\infty}_{\underline{h}(t)}\Big(J_1(x-y)\underline{u}(x,t)+\rho J_2(x-y)\underline{v}(x,t)\Big)\rd y \rd x\\
			\geq&\ \mu \rho  \int_{0}^{\underline{h}(t)}\int^{+\infty}_{\underline{h}(t)}J_2(x-y)\underline{v}(x,t)\rd y \rd x\geq \mu \rho\int_{0}^{\underline{h}(t)}\int^{+\infty}_{\underline{h}(t)}\delta_2J_2(x-y)\lf(\frac{\underline h(t)-x}{\underline h(t)}\rr)^\lambda \rd y\rd x\\
			= &\ \frac{\mu \rho}{\underline{h}^\lambda(t)}\int_{-\underline{h}(t)}^{0}\int^{+\infty}_{0}\delta_2J_2(x-y)(-x)^\lambda \rd y\rd x
			=\frac{\mu \rho}{\underline{h}^\lambda(t)}\int^{\underline{h}(t)}_{0}\int^{+\infty}_{x} \delta_2 J_2(y)x^\lambda dydx\\
			=&\ \frac{\mu \rho}{\underline{h}^\lambda(t)}\lf[\int^{\underline{h}(t)}_{0}\int^{y}_{0}+\int_{\underline{h}(t)}^{+\infty}
			\int^{\underline{h}(t)}_{0}
			\rr]\delta_2J_2(y)x^\lambda dxdy\\
			\geq&\ \frac{\mu \rho}{\underline{h}^\lambda(t)}\int^{\underline{h}(t)}_{0}\int^{y}_{0}\delta_2J_2(y)x^\lambda dxdy
			= \frac{\mu \rho}{(\lambda+1)\underline{h}^\lambda(t)}\int^{\underline{h}(t)}_{0}\delta_2J_2(y)y^{\lambda+1} \rd y.
			\end{align*}
By \eqref{J_1J_2}, when $\sigma_1\gg 1$,
			\begin{align*}
				&\frac{\mu\rho}{(\lambda+1)\underline{h}^\lambda(t)}\int^{\underline{h}(t)}_{0}\delta_2J_2(y)y^{\lambda+1} \rd y\\
				&\geq\frac{K_1\delta_2\mu\rho}{(\lambda+1)\underline{h}^\lambda(t)}\int^{\underline{h}(t)}_{1}y^{\lambda+1-\alpha} \rd y= 	\frac{K_1\delta_2\mu\rho (\underline h^{\lambda+2-\alpha}(t)-1)}{(\lambda+1)(\lambda+2-\alpha)\underline{h}^\lambda(t)}\\
				&\geq\frac{K_1\delta_2\mu\rho }{2(\lambda+1)(\lambda+2-\alpha)}\underline{h}^{2-\alpha}(t)= \frac{K_1\delta_2\mu\rho }{2(\lambda+1)(\lambda+2-\alpha)}(C_1t+\sigma_1)^{\frac{2-\alpha}{\alpha-1}}\\
			&\geq\frac{C_1}{\alpha-1}(C_1t+\sigma_1)^{\frac{2-\alpha}{\alpha-1}}=\underline{h}'(t)
		\end{align*}
	provided that
	\begin{align*}
		C_1\leq \frac{K_1\delta_2\mu\rho(\alpha-1) }{2(\lambda+1)(\lambda+2-\alpha)}.
	\end{align*}
 Thus \eqref{lower-ac1} holds for small $C_1$ and large $\sigma_1$.

{\bf Step 2}.	We prove the following inequalities for $t>0$ and $|x|\in [0, \underline h(t)]\setminus\{\psi_1(t)\}$,
\begin{equation}\label{lower2.4}
\begin{cases}
	\dd\underline u_t\leq d_1\int^{\underline h(t)}_{-\underline h(t)}J_1(x-y)\underline{u}(y,t)dy-d_1\underline{u}
	-a_{11}\underline{u}+a_{12}\int^{\underline h(t)}_{-\underline h(t)}K(x-y)\underline{v}(y,t)dy,\\[3mm]
	\dd\underline{v}_t\leq d_2\int^{\underline h(t)}_{-\underline h(t)}J_2(x-y)\underline{v}(y,t)dy-d_2\underline{v}
-a_{22}\underline{v}+G(\underline{u}).
\end{cases}
\end{equation}

We first calculate $\underline u_t$ and $\underline v_t$. 	It is easily seen that for $x\in [-\underline h(t),\underline h(t)]$ and $t\geq 0$,
	\begin{align*}
			\underline{u}_t(x,t)=\delta_1\lambda\lf(\frac{\underline h(t)-|x|}{\underline h(t)}\rr)^{\lambda-1}\frac{\underline{h}'(t)|x|}{\underline{h}^2(t)}=\frac{C_1\delta_1\lambda}{\alpha-1}\lf(\frac{\underline h(t)-|x|}{\underline h(t)}\rr)^{\lambda-1}\frac{|x|}{\underline{h}(t)}\underline{h}^{1-\alpha}(t),
	\end{align*}
where we have used $\underline h'(t)=\frac{C_1}{\alpha-1}\underline h^{2-\alpha}(t)$. Similarly,
\begin{align*}
	\underline{v}_t(x,t)=\frac{C_1\delta_2\lambda}{\alpha-1}\lf(\frac{\underline h(t)-|x|}{\underline h(t)}\rr)^{\lambda-1}\frac{|x|}{\underline{h}(t)}\underline{h}^{1-\alpha}(t)
	\mbox{ for } 0\leq |x|<\psi_1(t).
\end{align*}
For $\psi_1(t)< |x|<\underline h(t)$, using
\begin{align*}
\Psi_1=\frac{\underline h-\psi_1}{\underline h}=\frac{C_2\underline h^{1-\alpha}}{C_2\underline h^{1-\alpha}+1}\leq C_2\underline h^{1-\alpha},\ 	\ (\Psi_1)_t=\frac{C_2(1-\alpha)\underline h^{-\alpha}\underline h'}{(C_2\underline h^{1-\alpha}+1)^2}=-\frac{C_1C_2\underline h^{2-2\alpha}}{(C_2\underline h^{1-\alpha}+1)^2}<0,
\end{align*}
we obtain
\begin{equation}\begin{aligned}\label{3.6}
	\underline{v}_t(x,t)&
	=\delta_2(\Psi_1)_t\lf(\frac{\underline h-|x|}{\underline h}\rr)^{\lambda-1}+\Psi_1\frac{C_1\delta_2(\lambda-1)}{\alpha-1}\lf(\frac{\underline h-|x|}{\underline h}\rr)^{\lambda-2}\frac{|x|}{\underline{h}}\underline{h}^{1-\alpha}\\
	&\leq\Psi_1\frac{C_1\delta_2(\lambda-1)}{\alpha-1}\lf(\frac{\underline h-|x|}{\underline h}\rr)^{\lambda-2}\frac{|x|}{\underline{h}}\underline{h}^{1-\alpha}\\
	&\leq\frac{C_1C_2\delta_2(\lambda-1)}{\alpha-1}\lf(\frac{\underline h-|x|}{\underline h}\rr)^{\lambda-2}\underline{h}^{2-2\alpha}.
\end{aligned}
\end{equation}

We next show that $\underline u_t$ and $\underline v_t$ can be bounded from above by constant multiples of  $\underline u$, $\underline v$ and  $\int_{-\underline h}^{\underline h} J_2(x-y) \underline v(y,t) \rd y$; namely,  \medskip

{\bf Claim 1}: For any fixed $\epsilon>0$ we can find suitable  $C_1$ and $C_2$ such that for  $\sigma_1\gg1$, $t\geq 0$ and $|x|\in [0, \underline h(t)]\setminus\{\psi_1(t)\}$,
\begin{align}\label{lower2.5}
	\underline u_t\leq \epsilon (\underline u+\underline v),\ \ \underline v_t\leq \epsilon \underline v+\epsilon \int_{-\underline h}^{\underline h} J_2(x-y) \underline v(y,t) \rd y.
\end{align}

The proof of \eqref{lower2.5} will be carried out according to two cases: $ |x|< \psi_1(t)$ and $\psi_1(t)< |x|< \underline h (t)$. 

 For
$
	|x|< \psi_1 (t)=\frac{\underline h (t)}{C_2\underline{h}^{1-\alpha}(t)+1}$,
by the earlier computations,
\begin{align*}
	\underline u_t- \epsilon \underline u&=\frac{C_1\delta_1\lambda}{\alpha-1}\lf(\frac{\underline h-|x|}{\underline h}\rr)^{\lambda-1}\frac{|x|}{\underline{h}}\underline{h}^{1-\alpha}-\epsilon \lf(\frac{\underline h-|x|}{\underline h}\rr)^{\lambda}\delta_1\\
	&=\epsilon \lf(\frac{\underline h-|x|}{\underline h}\rr)^{\lambda-1}\delta_1\lf[\frac{C_1\lambda}{\epsilon (\alpha-1)}\frac{|x|}{\underline{h}}\underline{h}^{1-\alpha}+\frac{|x|}{\underline{h}}-1\rr]\\
	&=\epsilon \lf(\frac{\underline h-|x|}{\underline h}\rr)^{\lambda-1}\delta_1\lf[\lf(\frac{C_1\lambda}{\epsilon (\alpha-1)}\underline{h}^{1-\alpha}+1\rr)\frac{|x|}{\underline{h}}-1\rr]\\
	&\leq \epsilon \lf(\frac{\underline h-|x|}{\underline h}\rr)^{\lambda-1}\delta_1\lf[\lf(\frac{C_1\lambda}{\epsilon (\alpha-1)}\underline{h}^{1-\alpha}+1\rr)\frac{1}{C_2\underline{h}^{1-\alpha}+1}-1\rr]
	\leq 0
\end{align*}
provided that
\begin{align*}
	C_2\geq \frac{C_1\lambda}{\epsilon (\alpha-1)}.
\end{align*}
The same calculation also yields, under the same conditions for $C_1$ and $C_2$,
\begin{align*}
	\underline v_t-\epsilon \underline v\leq 0 \mbox{ for } |x|< \psi_1(t).
\end{align*}
Hence \eqref{lower2.5} holds  for $|x|<\psi_1 (t)$.

Next we check \eqref{lower2.5} for $\psi_1 (t)< |x|< \underline h (t)$.   Clearly,
\begin{align*}
	\underline u_t=&\frac{C_1\delta_1\lambda}{\alpha-1}\lf(\frac{\underline h-|x|}{\underline h}\rr)^{\lambda-1}\frac{|x|}{\underline{h}}\underline{h}^{1-\alpha}\leq \frac{C_1\delta_1\lambda}{\alpha-1}\lf(\frac{\underline h-|x|}{\underline h}\rr)^{\lambda-1}\underline{h}^{1-\alpha},
\end{align*}
and from the definitions of $\underline v$ and $\Psi_1$, we have
\begin{align*}
	 \underline v=&  \delta_2\Psi_1\lf(\frac{\underline h-|x|}{\underline h}\rr)^{\lambda-1}=\delta_2\frac{C_2\underline h^{1-\alpha}}{C_2\underline h^{1-\alpha}+1}\lf(\frac{\underline h-|x|}{\underline h}\rr)^{\lambda-1}\geq \frac{\delta_2C_2\underline h^{1-\alpha}}{2}\lf(\frac{\underline h-|x|}{\underline h}\rr)^{\lambda-1}
\end{align*}
for  large $\underline h$ which can be guaranteed by enlarging $\sigma_1$ if necessary. Therefore
\begin{align*}
\underline u_t\leq \epsilon \underline v\leq \epsilon (\underline u+\underline v) \ \mbox{ for }\ \psi_1(t)< |x|< \underline h (t)
\end{align*}
provided that
\begin{align*}
	C_2\geq \frac{2C_1\delta_1\lambda}{\epsilon\delta_2} \mbox{ and } \sigma_1\gg 1.
\end{align*}

It remains to verify \eqref{lower2.5} for $v_t$ for  $\psi_1(t)\leq |x|\leq \underline h(t)$. By \eqref{3.6},
\begin{align*}
	\underline v_t&\leq\frac{C_1C_2\delta_2(\lambda-1)}{\alpha-1}\lf(\frac{\underline h-|x|}{\underline h}\rr)^{\lambda-2}\underline{h}^{2-2\alpha}\\
	&\leq \frac{C_1C_2\delta_2(\lambda-1)}{\alpha-1}\lf(\frac{\underline h-\psi_1}{\underline h}\rr)^{\lambda-2}\underline{h}^{2-2\alpha}\\
	&=\frac{C_1C_2^{\lambda-1}\delta_2(\lambda-1)}{\alpha-1}\frac{\underline h^{\lambda(1-\alpha)}}{\lf(C_2\underline{h}^{1-\alpha}+1\rr)^{\lambda-2}}\leq \frac{C_1C_2^{\lambda-1}\delta_2(\lambda-1)}{\alpha-1}\underline h^{\lambda(1-\alpha)}.
\end{align*}
On the other hand,  since $J_2$ is even,
\begin{align*}
	\int_{-\underline h}^{\underline h} J_2(x-y) \underline v(y,t) \rd y=\int_{-\underline h}^{\underline h} J_2(y-x) \underline v(y,t) \rd y=\int_{-\underline h-x}^{\underline h-x} J_2(y) \underline v(x+y,t) \rd y.
\end{align*}
Moreover, since $\psi_1/\underline h$ is close to  $1$ for large $\underline h$, we have
\begin{align*}
 &(-\underline h/2, -\underline h/4)	\subset (-\underline h-x,\underline h-x) \ \mbox{ for }\ \psi_1\leq x\leq \underline h,\\
 &x+y\in [\underline h/4, 3\underline h /4]   \ \mbox{ for }\ \psi_1\leq x\leq \underline h,\ y\in(-\underline h/2, -\underline h/4),
\end{align*}
and so by \eqref{J_1J_2},
\begin{equation}\label{3.8}
\begin{aligned}
	&\int_{-\underline h}^{\underline h} J_2(x-y) \underline v(y,t) \rd y\geq \int_{-\underline h/2}^{-\underline h/4} J_2(y) \underline v(x+y,t) \rd y\\
	&\geq K_1\delta_2\int_{-\underline h/2}^{-\underline h/4} |y|^{-\alpha} \lf(\frac{\underline h-|x+y|}{\underline h}\rr)^{\lambda} \rd y=K_1\delta_2\int_{-\underline h/2}^{-\underline h/4} |y|^{-\alpha} \lf(\frac{\underline h-x-y}{\underline h}\rr)^{\lambda} \rd y\\
	&\geq	\frac{K_1\delta_2}{4^{\lambda}}\int_{-\underline h/2}^{-\underline h/4} |y|^{-\alpha}  \rd y=\frac{K_1\delta_2(4^{\alpha-1}-2^{\alpha-1})}{4^{\lambda}(\alpha-1)} \underline h^{1-\alpha}.
\end{aligned}
\end{equation}
Hence, thanks to
\begin{align*}
 \alpha< \lambda (\alpha-1)+1,
\end{align*}
we deduce for large $\underline h$,
\begin{align*}
	\underline v_t\leq \frac{C_1C_2^{\lambda-1}\delta_2(\lambda-1)}{\alpha-1}\underline h^{\lambda(1-\alpha)}<\epsilon \frac{K_1\delta_2(4^{\alpha-1}-2^{\alpha-1})}{4^{\lambda}(\alpha-1)} \underline h^{1-\alpha}\leq \epsilon\int_{-\underline h}^{\underline h} J_2(x-y) \underline v(y,t) \rd y,
\end{align*}
which yields \eqref{lower2.5} for $\psi_1(t)<x<\underline h(t)$. The proof for $\psi_1 (t)<-x<\underline h (t)$ is similar. Claim 1 is now proved, after checking that all the requirements for $C_1, C_2$ and $\sigma_1$ stated above can be met simultaneously.

In view of Claim 1,  in order to verify \eqref{lower2.4}, it remains to show that
\begin{equation}\begin{aligned}\label{lower2.6}
		& d_1\int^{\underline h(t)}_{-\underline h(t)}J_1(x-y)\underline{u}(y,t)dy-d_1\underline{u}
	-a_{11}\underline{u}+a_{12}\int^{\underline h(t)}_{-\underline h(t)}K(x-y)\underline{v}(y,t)dy,\\[2mm]
	&\geq  \epsilon\lf(\underline u+\underline v\rr)
\end{aligned}
\end{equation}
and
\begin{equation}\begin{aligned}\label{lower2.7}
		 &d_2\int^{\underline h(t)}_{-\underline h(t)}J_2(x-y)\underline{v}(y,t)dy-d_2\underline{v}
	-a_{22}\underline{v}+G(\underline{u}) \\
	&\geq \epsilon\lf(\underline v+\int^{\underline h(t)}_{-\underline h(t)}J_2(x-y)\underline{v}(y,t)dy\rr).
\end{aligned}
\end{equation}

By \eqref{linear} and the fact that
\begin{align*}
	(\underline{u},\underline{v})=\lf(\frac{\underline h(t)-|x|}{\underline h(t)}\rr)^{\lambda}(\delta_1,\delta_2)\ \mbox{ for }\ |x|\leq  \psi_1(t),
\end{align*}
 we have
	\begin{equation*}
		\begin{cases}
			-a_{11}\underline{u}+a_{12}\underline{v}\geq  \rho (\underline{u}+\underline{v}), &|x|\leq \psi_1(t),\\
			-a_{22}\underline{v}+G(\underline{u})\geq  \rho (\underline{u}+\underline{v}),&|x|\leq \psi_1(t).
		\end{cases}
	\end{equation*}
 For $|x|\geq  \psi_1(t)$,  we have  $\underline u\leq s\delta_1,\ \underline v=s\delta_2$ with
 \[
 s:= \Psi_1(x,t)\lf(\frac{\underline h(t)-|x|}{\underline h(t)}\rr)^{\lambda-1}.
 \]
  So by \eqref{linear}, 
\begin{align*}
	-a_{11}\underline{u}+a_{12}\underline{v}-  \rho (\underline{u}+\underline{v})\geq s[-a_{11}\delta_1+a_{12}\delta_2-\rho(\delta_1+\delta_2)]\geq 0 \ \mbox{ for }  \psi_1(t)\leq |x|\leq \underline h(t).
\end{align*}
By Example \ref{example2} (a) and (b) with $L(t)$ replaced by $\underline h(t)$, we see for any given small $\epsilon>0$, there is large $h_*=h_*(\epsilon)$ such that for $\underline h(t)\geq h_*$ and $|x|\leq \underline h(t)$,
\begin{align*}
	&\int^{\underline h(t)}_{-\underline h(t)}J_1(x-y)\underline{u}(y,t)dy\geq (1-\epsilon)\underline u(x,t),\\
	&\int^{\underline h(t)}_{-\underline h(t)}J(x-y)\underline{v}(y,t)dy\geq (1-\epsilon)\underline v(x,t) \mbox{ for } J\in\{J_2, K\},
\end{align*}
from which we can easily obtain \eqref{lower2.6} for $|x|\leq \underline h(t)$ and  \eqref{lower2.7} for  $|x|\leq \psi_1(t)$. Indeed, the left hand side of \eqref{lower2.6}
is bounded from below by
\[
d_1(1-\epsilon)\underline u-d_1\underline u+\rho(\underline u+\underline v)-a_{12}\epsilon\underline v\geq (\rho/2)(\underline u+\underline v)\geq \epsilon(\underline u+\underline v)
\]
if $\epsilon>0$ is small enough so that $(d_1+a_{12}+1)\epsilon\leq \rho/2$. For \eqref{lower2.7}, we note that it follows from
\[\begin{aligned}
&(d_2-\epsilon)\int^{\underline h(t)}_{-\underline h(t)}J_2(x-y)\underline{v}(y,t)dy-d_2\underline{v}
	-a_{22}\underline{v}+G(\underline{u})\\
	&\geq (d_2-\epsilon)(1-\epsilon)\underline v-d_2\underline v+\rho(\underline u+\underline v)
	\geq (\rho/2)(\underline u+\underline v)\geq \epsilon\underline v
\end{aligned}\]
provided that $(d_2+1)\epsilon\leq \rho/2$.

The proof of \eqref{lower2.7} for  $\psi_1(t)\leq |x|\leq \underline h(t)$ uses the following conclusion.

{\bf Claim 2}.   For any given constant $M>0$, there exists $M_0>0$ such that
\begin{align}\label{3.11}
	\int^{\underline h(t)}_{-\underline h(t)}J_2(x-y)\underline{v}(y,t)dy\geq M \underline v(x,t) \mbox{ for } \psi_1(t)\leq |x|\leq \underline h(t),\ t\geq 0,
\end{align}
whenever $\sigma_1\geq M_0$.

This, combined with $G(\underline{u})\geq  0$, clearly yields
 \eqref{lower2.7} for  $\psi_1(t)\leq |x|\leq \underline h(t)$. Thus to  finish the proof of Step 2, it suffices to prove \eqref{3.11}.

 For $\psi_1(t)\leq |x|\leq \underline h(t)$, by the definition of  $\underline v$ ,
\begin{align*}
	\underline v=   \delta_2\frac{\underline h-\psi_1}{\underline h}\lf(\frac{\underline h-|x|}{\underline h}\rr)^{\lambda-1}
	\leq\delta_2	\lf(\frac{\underline h-\psi_1}{\underline h}\rr)^{\lambda}=\delta_2\lf(\frac{C_2\underline h^{1-\alpha}}{C_2\underline h^{1-\alpha}+1}\rr)^\lambda
	\leq \delta_2C_2^{\lambda}\underline h^{\lambda(1-\alpha)}.
\end{align*}
On the other hand, by \eqref{3.8},
\begin{align*}
&\int_{-\underline h}^{\underline h} J_2(x-y) \underline v(y,t) \rd y\geq \frac{K_1\delta_1(4^{\alpha-1}-2^{\alpha-1})}{4^{\lambda}(\alpha-1)} \underline h^{1-\alpha}.
\end{align*}
These calculations together with $\alpha>1$ and $\lambda>1$ give \eqref{3.11} for large $\underline h$. We have now finished the proof of Claim 2.

{\bf Step 3}. Completion of the proof by the comparison principle.
			
Since spreading happens and $u^*>\delta_1\geq \underline u$, $v^*>\delta_2\geq \underline v$, there exists $t_0>0$ large enough such that $[g(t_0),h(t_0)]\supset[-\underline{h}(0),\underline{h}(0)]$, and also
\begin{align*}
u(x,t_0)\geq \delta_1\geq \underline{u}(x,0),~~v(x,t_0)\geq \delta_2\geq \underline{v}(x,0) \mbox{ for } 	x\in[-\underline{h}(0),\underline{h}(0)].
\end{align*}
 By definition, $(\underline{u}(x,t),\underline{v}(x,t))=(0,0)$ for $x=\pm \underline{h}(t)$ and $t\geq0$.

In view of the inequalities proved in Steps 1 and 2, we are now in a position to apply a suitable version of the comparison principle (see Remark 2.4 in \cite{WangDu-JDE}, and also Remark 2.4 in \cite{dn2022} which explains why the jumping discontinuity of $\underline v_t$ along $|x|=\psi_1(t)$ does not affect the conclusion) to assert that
				\[
				-\underline h(t)\geq g(t_0+t),
		 \ \ \underline{h}(t)\leq h(t_0+t) \mbox{ for }t\geq0.
				\]
The proof is finished.
			\end{proof}

\begin{lemma}\label{alpha2}
Suppose that $(\mathbf{J})$ holds and $\alpha=2$ in \eqref{J_1J_2}. If spreading occurs to \eqref{1}, then there exists $D>0$ such that
\begin{equation*}
-g(t), h(t)\geq D\,t\ln t\mbox{ for all large } t>0.\\
\end{equation*}

\end{lemma}

\begin{proof}
Fix constants $\beta\in(0,1/2)$ and $\lambda> \frac{1}{\beta}>2$. Then define
\begin{equation*}
\begin{cases}
\underline{h}(t):=C_1(t+\sigma_2)\ln (t+\sigma_2),& t\geq0,\\
	\underline{u}(x,t):=\min\big\{1,\lf(\frac{\underline{h}(t)-|x|}{(t+\sigma_2)^{\beta}}\rr)^{\lambda}\big\}
	\delta_1, & x\in[-\underline{h}(t),\underline{h}(t)],~t\geq0,\\
	\underline{v}(x,t):=\min\big\{1,\Psi_2(t,x)\lf(\frac{\underline{h}(t)-|x|}{(t+\sigma_2)^{\beta}}\rr)^{\lambda-1}\big\}
	\delta_2, & x\in[-\underline{h}(t),\underline{h}(t)],~t\geq0,
	\end{cases}
\end{equation*}
where
\begin{equation*}
	\Psi_2(x,t):=
	\begin{cases}
		\dd	\frac{\underline h(t)-|x|}{(t+\sigma_2)^{\beta}},&|x|\leq \psi_2(t),\\[3mm]
		\dd	\frac{\underline h(t)-\psi_2(t)}{(t+\sigma_2)^{\beta}},& |x|\geq \psi_2(t),
	\end{cases}\ \ \  \ \psi_2(t):=\dd\frac{\underline h(t)-C_3\ln (t+\sigma_2)}{C_2(t+\sigma_2)^{-1}+1},
\end{equation*}
$(\delta_1,\delta_2)$ satisfies \eqref{linear}, \eqref{linear1},  while $C_1$, $C_2$, $C_3$
and $\sigma_2$ are positive constants to be determined.

We first show that, with suitable choices of the parameters,
\begin{equation}\label{lower-ac4}
\begin{aligned}
\underline{h}'(t)\leq\mu \int_{-\underline{h}(t)}^{\underline{h}(t)}\int^{+\infty}_{\underline{h}(t)}\Big(J_1(x-y)\underline{u}(x,t)+\rho J_2(x-y)\underline{v}(x,t)\Big)\rd y \rd x
\mbox{ for $t>0$,}
\end{aligned}
\end{equation}
which, combined with the facts that $(\underline{u}(x,t),\underline{v}(x,t))=(\underline{u}(-x,t),\underline{v}(-x,t))$
and $J_i(x)=J_i(-x)$ for $i=1,2$ and $x\in[-\underline{h}(t), \underline{h}(t)]$,  implies
\begin{align*}
	-\underline{h}'(t)\geq-\mu \int_{-\underline{h}(t)}^{\underline{h}(t)}\int_{-\infty}^{-\underline{h}(t)}\Big(J_1(x-y)\underline{u}(x,t)+\rho J_2(x-y)\underline v (x,t)\Big)\rd y \rd x
	\ \mbox{ for } t>0.
\end{align*}

By  the definitions of $\underline u$ and $\underline v$, we have
\begin{align*}
&\mu \int_{-\underline{h}(t)}^{\underline{h}(t)}\int^{+\infty}_{\underline{h}(t)}\Big(J_1(x-y)\underline{u}(x,t)+\rho J_2(x-y)\underline{v}(x,t)\Big) \rd y \rd x\\
&\geq \mu \int_{0}^{\underline{h}(t)-(t+\sigma_2)^{\beta}}\int^{+\infty}_{\underline{h}(t)}\rho J_2(x-y)\underline{v}(x,t) \rd y \rd x\\
&\geq \delta_2 \mu \int_{0}^{\underline{h}(t)-(t+\sigma_2)^{\beta}}\int^{+\infty}_{\underline{h}(t)}\rho J_2(x-y) \rd y \rd x\\
&=\delta_2 \mu \int_{-\underline{h}(t)}^{-(t+\sigma_2)^{\beta}}\int^{+\infty}_{0}\rho J_2(x-y)\rd y \rd x=\delta_2 \mu \rho \int_{(t+\sigma)^{\beta}}^{\underline{h}(t)}\int^{+\infty}_{x} J_2(y)\rd y \rd x.
	\end{align*}
By changing the order of integration, and then using \eqref{J_1J_2}, we obtain, for $\sigma_2\gg 1$ and $t>0$,
\begin{align*}
	&\int_{(t+\sigma)^{\beta}}^{\underline{h}(t)}\int^{+\infty}_{x} J_2(y)\rd y \rd x\\
	&= \Big[\int^{\underline{h}(t)}_{(t+\sigma_2)^{\beta}}\int^y_{(t+\sigma_2)^{\beta}}+\int
	^{\infty}_{\underline{h}(t)}\int^{\underline{h}(t)}_{(t+\sigma_2)^{\beta}}\Big] J_2(y) \rd x \rd y\\
	&\geq  \int^{\underline{h}(t)}_{(t+\sigma_2)^{\beta}}\int^y_{(t+\sigma_2)^{\beta}} J_2(y) \rd x \rd y\\
	&\geq K_1\int^{\underline{h}(t)}_{(t+\sigma_2)^{\beta}}\frac{y-(t+\sigma_2)^{\beta}}{y^2}dy\\
	&= K_1\Big[\ln(\underline{h}(t))-\beta\ln(t+\sigma_2)+\frac{(t+\sigma_2)^{\beta}}{\underline{h}(t)}-1\Big]\\	
	&\geq K_1\Big[\ln(\underline{h}(t))-\beta\ln(t+\sigma_2)-1\Big]\\
	&\geq  K_1(1-\beta)\big[\ln(t+\sigma_2)+1\big].
\end{align*}
Therefore
\begin{align*}
&\mu \int_{-\underline{h}(t)}^{\underline{h}(t)}\int^{+\infty}_{\underline{h}(t)}\Big(J_1(x-y)\underline{u}(x,t)+\rho J_2(x-y)\underline{v}(x,t)\Big) \rd y \rd x\\
\geq &\ \delta_2 \mu K_1\rho(1-\beta)\big[\ln(t+\sigma_2)+1\big]\\
\geq&\ C_1\big[\ln(t+\sigma_2)+1\big]=\underline{h}'(t) \mbox{ for } t>0
\end{align*}
provided that $\sigma_2\gg 1$ and
\begin{align*}
	C_1\leq \delta_2\mu K_1\rho(1-\beta).
\end{align*}
This proves \eqref{lower-ac4}.

Next we show that by adding further restrictions on the parameters, we have, for $|x|\in [0, \underline h(t)]\setminus \{\underline h(t)- (t+\sigma_2)^{\beta},\ \psi_2(t)\}$ and $t>0$,
\begin{equation}\label{3.14}
	\begin{cases}
		\dd\underline u_t\leq d_1\int^{\underline h(t)}_{-\underline h(t)}J_1(x-y)\underline{u}(y,t)dy-d_1\underline{u}
		-a_{11}\underline{u}+a_{12}\int^{\underline h(t)}_{-\underline h(t)}K(x-y)\underline{v}(y,t)dy,\\[3mm]
		\dd\underline{v}_t\leq d_2\int^{\underline h(t)}_{-\underline h(t)}J_2(x-y)\underline{v}(y,t)dy-d_2\underline{v}
		-a_{22}\underline{v}+G(\underline{u}).
	\end{cases}
\end{equation}

We will establish \eqref{3.14} in three steps, where $\sigma_2\gg 1$ is always assumed.

{\bf Step 1}. Preliminary estimates  of $\underline u_t$ and $\underline v_t$.

Note that   $\underline u$ and $\underline v$ are piece-wisely differentiable with respect to $t$, with $\underline u_t$ having a jumping discontinuity at $|x|=\underline h(t)-(t+\sigma_2)^\beta$, and $\underline v_t$ having jumping discontinuities at $|x|=\underline h(t)-(t+\sigma_2)^\beta$ and at $|x|=\psi_2(t)$.
Let us note that $\sigma_2\gg 1$ implies
\[
\underline h(t)-(t+\sigma_2)^\beta<\psi_2(t) \mbox{ for } t>0.
\]
  Then clearly,
\begin{align*}
	\underline{u}_t(x,t)=\underline{v}_t(x,t)=0\ \mbox{ for } \ 0\leq |x|<\underline{h}(t)-(t+\sigma_2)^{\beta}.
\end{align*}
For $\underline{h}(t)-(t+\sigma_2)^{\beta}<|x|\leq \underline{h}(t)$,
\begin{equation}
\begin{aligned}\label{3.16u}
\underline{u}_t(x,t)=&\ \lambda \delta_1\lf(\frac{\underline{h}(t)-|x|}{(t+\sigma_2)^{\beta}}\rr)^{\lambda-1}\lf(\frac{C_1\big[(1-\beta)\ln(t+\sigma_2)+1\big]}{(t+\sigma_2)^{\beta}}
+\frac{\beta|x|}{(t+\sigma_2)^{\beta+1}}\rr)\\
\leq&\ \lambda \delta_1\lf(\frac{\underline{h}(t)-|x|}{(t+\sigma_2)^{\beta}}\rr)^{\lambda-1}\lf(\frac{C_1\ln(t+\sigma_2)}{(t+\sigma_2)^{\beta}}
+\frac{|x|}{(t+\sigma_2)^{\beta+1}}\rr).
\end{aligned}
\end{equation}
For  $\underline{h}(t)-(t+\sigma_2)^{\beta}<|x|< \psi_2(t)$, $\underline v$ is a constant multiple of $\underline u$, and so
\begin{equation*}
	\underline{v}_t(x,t)\leq \lambda \delta_2\lf(\frac{\underline{h}(t)-|x|}{(t+\sigma_2)^{\beta}}\rr)^{\lambda-1}\lf(\frac{C_1\ln(t+\sigma_2)}{(t+\sigma_2)^{\beta}}
	+\frac{|x|}{(t+\sigma_2)^{\beta+1}}\rr).
 \end{equation*}
For $\psi_2(t)<|x|\leq \underline{h}(t)$ and $\sigma_2\gg 1$,
\begin{align*}
	\Psi_2(x, t)=&\ \frac{\underline h(t)-\psi_2(t)}{(t+\sigma_2)^{\beta}}=\dd\frac{(C_1C_2+ C_3)\ln (t+\sigma_2)}{(t+\sigma_2)^{\beta}[C_2(t+\sigma_2)^{-1}+1]}\leq (C_1C_2+C_3)\frac{\ln (t+\sigma_2)}{(t+\sigma_2)^\beta},\label{Psi2}\\
	 (\Psi_2)_t(x,t)=&\ \dd (C_1C_2+C_3)\frac{ [1-\beta \ln (t+\sigma_2)] (t+\sigma_2) +[C_2+(1-\beta) \ln (t+\sigma_2)]}{(t+\sigma_2)^{\beta+2}[C_2(t+\sigma_2)^{-1}+1]^2}< 0,\nonumber
\end{align*}
and therefore,
\begin{equation}\begin{aligned}\label{3.15}
	\underline{v}_t(x,t)=&\ (\Psi_2)_t \delta_2\lf(\frac{\underline{h}(t)-|x|}{(t+\sigma_2)^{\beta}}\rr)^{\lambda-1}\\
	&+\Psi_2\delta_2(\lambda-1)\lf(\frac{\underline{h}(t)-|x|}{(t+\sigma_2)^
{\beta}}\rr)^{\lambda-2}\lf(\frac{C_1\big[(1-\beta)\ln(t+\sigma_2)+1\big]}{(t+\sigma_2)^{\beta}}
	+\frac{\beta |x|}{(t+\sigma_2)^{\beta+1}}\rr)\\
	\leq &\ \Psi_2\delta_2\lambda \lf(\frac{\underline{h}(t)-|x|}{(t+\sigma_2)^{\beta}}\rr)^{\lambda-2}\lf(\frac{C_1\ln(t+\sigma_2)}{(t+\sigma_2)^{\beta}}
	+\frac{|x|}{(t+\sigma_2)^{\beta+1}}\rr)\\
	\leq &\ \delta_2\lambda \Psi_2^{\lambda-1}\lf(\frac{C_1\ln(t+\sigma_2)}{(t+\sigma_2)^{\beta}}
	+\frac{\underline h(t)}{(t+\sigma_2)^{\beta+1}}\rr)\\
	\leq&\ 2C_1\delta_2\lambda\left(C_1C_2+C_3\right)^{\lambda-1}\left(\frac{\ln (t+\sigma_2)}{(t+\sigma_2)^{\beta}}\right)^\lambda.
\end{aligned}
\end{equation}

With these preparation, we are ready to obtain some easy to use upper bounds for $\underline u_t$ and $\underline v_t$ in the next step.

{\bf Step 2}. For any fixed $C>0$ we can find  suitable parameters $C_1, C_2, C_3$ and $\sigma_2$ such  that for  $t> 0$ and $|x|\in [0, \underline h(t)]\setminus \{\underline h(t)- (t+\sigma_2)^{\beta},\psi_2(t)\}$,
\begin{align}\label{3.17}
	\underline u_t(x,t)\leq  C\max\{\underline u(x,t),\underline v(x,t)\},\ \ \underline v_t(x,t)\leq  C\max\lf\{\underline v(x,t),\int_{-\underline h (t)}^{\underline h(t)} J_2(x-y) \underline v(y,t) \rd y\rr\}.
\end{align}

The proof of \eqref{3.17} will be carried out according to three cases: $0\leq |x|<\underline{h}(t)-(t+\sigma_2)^{\beta}$, $\underline{h}(t)-(t+\sigma_2)^{\beta}<|x|<
\psi_2(t)$ and $\psi_2(t)<|x|<\underline h (t)$.

For $0\leq |x|<\underline{h}(t)-(t+\sigma_2)^{\beta}$, we have $\underline u_t=\underline v_t=0$ and hence  \eqref{3.17} holds trivially.

 For $\underline{h}(t)-(t+\sigma_2)^{\beta}<|x|<\psi_2(t)$, by \eqref{3.16u} we have, using $\sigma_2\gg 1$,
 \begin{align*}
 	&\underline u_t-C\underline u\\
 	\leq &\
  \lambda \delta_1\lf(\frac{\underline{h}(t)-|x|}{(t+\sigma_2)^{\beta}}\rr)^{\lambda-1}\lf(\frac{C_1\ln(t+\sigma_2)}{(t+\sigma_2)^{\beta}}
 	+\frac{|x|}{(t+\sigma_2)^{\beta+1}}\rr)-C\delta_1\lf(\frac{\underline{h}(t)-|x|}{(t+\sigma_2)^{\beta}}\rr)^{\lambda}\\
 	\leq &\ \delta_1\lf(\frac{\underline{h}(t)-|x|}{(t+\sigma_2)^{\beta}}\rr)^{\lambda-1}\lf(\frac{\lambda C_1\ln(t+\sigma_2)}{(t+\sigma_2)^{\beta}}
 	+\frac{\lambda \underline h(t)}{(t+\sigma_2)^{\beta+1}}-C\frac{\underline h(t)-\psi_2(t)}{(t+\sigma_2)^{\beta}}\rr)\\
 	=&\ \frac{\delta_1}{(t+\sigma_2)^\beta}\lf(\frac{\underline{h}(t)-|x|}{(t+\sigma_2)^{\beta}}\rr)^{\lambda-1}[2\lambda C_1\ln(t+\sigma_2)-C\frac{(C_1C_2+C_3)\ln (t+\sigma_2)}{C_2(t+\sigma_2)^{-1}+1}]\\
	\leq& \ \frac{\delta_1}{(t+\sigma_2)^\beta}\lf(\frac{\underline{h}(t)-|x|}{(t+\sigma_2)^{\beta}}\rr)^{\lambda-1}\ln (t+\sigma_2)[2\lambda C_1-C(C_1C_2+C_3)/2]<0
 \end{align*}
provided that $\sigma_2\gg 1$ and
\[
\frac{4\lambda C_1}{C_1C_2+C_3}<C.
\]
Noting that in the range $\underline{h}(t)-(t+\sigma_2)^{\beta}<|x|<\psi_2(t)$, the function $\underline v(t,x)$ is a constant multiple of $\underline u(t,x)$,
we similarly have
\[
\underline v_t-C\underline v\leq 0.
\]

For $\psi_2(t)<|x|<\underline h(t)$, by \eqref{3.16u} and the definitions of  $\underline v$ and $\Psi_2$, we obtain
\begin{align*}
	&\underline u_t-  C\max\{\underline u,\underline v\}\leq \underline u_t-  C\underline v\\
	\leq &\ \lambda \delta_1\lf(\frac{\underline{h}-|x|}{(t+\sigma_2)^{\beta}}\rr)^{\lambda-1}\lf(\frac{C_1\ln(t+\sigma_2)}{(t+\sigma_2)^{\beta}}
	+\frac{ \underline h}{(t+\sigma_2)^{\beta+1}}\rr)\\
	&-C\delta_2\lf(\frac{\underline{h}-|x|}{(t+\sigma_2)^{\beta}}\rr)^{\lambda-1}\dd\frac{ (C_1C_2+C_3)\ln (t+\sigma_2)}{(t+\sigma_2)^{\beta}[C_2(t+\sigma_2)^{-1}+1]}\\
	=&\ \frac{ \ln (t+\sigma_2)}{(t+\sigma_2)^{\beta}}\lf(\frac{\underline{h}-|x|}{(t+\sigma_2)^{\beta}}\rr)^{\lambda-1}\lf[2\lambda \delta_1 C_1
	-C\frac{ (C_1C_2+C_3)\delta_2}{C_2(t+\sigma_2)^{-1}+1}\rr]\\
	\leq&\ \frac{ \ln (t+\sigma_2)}{(t+\sigma_2)^{\beta}}\lf(\frac{\underline{h}-|x|}{(t+\sigma_2)^{\beta}}\rr)^{\lambda-1}\lf[2\lambda \delta_1 C_1
	-C(C_1C_2+C_3)\delta_2/2\rr]\leq 0
\end{align*}
 if
 \begin{align*}
 	\frac{4\lambda  C_1}
	{C_1C_2+C_3}<\frac{\delta_2}{\delta_1}C.
 \end{align*}

It remains to prove the desired upper bound for $\underline v_t$ for $\psi_2(t)<x<\underline h(t)$.
Note that $\frac{\psi_2(t)}{\underline h (t)}\to 1$ uniformly for $t>0$ as $\sigma_2\to \infty$.  It follows that for $\psi_2(t)<x<\underline h(t)$ and $\sigma_2\gg 1$,
\begin{align*}
	&\int_{-\underline{h}(t)}^{\underline{h}(t)}J_2(x-y)\underline{v}(y,t)\rd y=\int_{-\underline{h}(t)-x}^{\underline{h}(t)-x}J_2(y)\underline{v}(x+y,t)\rd y\\
	\geq& \int_{-\underline h(t)/2}^{-\underline{h}(t)/4}J_2(y)\underline{v}(x+y,t)\rd y\geq\delta_2 \int_{-\underline h(t)/2}^{-\underline{h}(t)/4}J_2(y)\rd y
\end{align*}
since $|x+y|\leq 4\underline h(t)/5$ for $x$ close to $\underline h(t)$ and  $y\in [-\underline h(t)/2,-\underline h(t)/4]$,  which leads to $\underline v(x+y,t)=\delta_2$.
We may now use \eqref{J_1J_2} to obtain
\begin{equation*}\begin{aligned}
	&\int_{-\underline{h}(t)}^{\underline{h}(t)}J_2(x-y)\underline{v}(y,t)\rd y\geq \delta_2 \int_{-\underline h(t)/2}^{-\underline{h}(t)/4}J_2(y)\rd y\\
	&\geq K_1\delta_2 \int_{-\underline h(t)/2}^{-\underline{h}(t)/4}y^{-2}\rd y=\frac{2K_1\delta_2}{\underline h(t)}.
\end{aligned}\end{equation*}
Therefore, for $\psi_2(t)<x<\underline h(t)$ and $\sigma_2\gg 1$, by \eqref{3.15},
\begin{align*}
	&\underline v_t-C\int_{-\underline{h}(t)}^{\underline{h}(t)}J_2(x-y)\underline{v}(y,t)\rd y\\
	&\leq 2C_1\delta_2\lambda\left(C_1C_2+C_3\right)^{\lambda-1}\left(\frac{\ln (t+\sigma_2)}{(t+\sigma_2)^{\beta}}\right)^\lambda
	-C\frac{2K_1\delta_2}{\underline h(t)}\\
	&= 2C_1\delta_2\lambda\left(C_1C_2+C_3\right)^{\lambda-1}\left(\frac{\ln (t+\sigma_2)}{(t+\sigma_2)^{\beta}}\right)^\lambda
	-\frac{2CK_1\delta_2}{C_1(t+\sigma_2)\ln (t+\sigma_2)}\\
	&<0\ \mbox{ since $ \lambda\beta >1$. }
\end{align*}

{\bf Step 3}. We prove \eqref{3.14}.

In view of \eqref{3.17} proved in
 Step 2, we only need to find a constant  $\widetilde C>0$ such that
\begin{align}\label{3.18}
	 d_1\int^{\underline h(t)}_{-\underline h(t)}J_1(x-y)\underline{u}(y,t)dy-d_1\underline{u}
	-a_{11}\underline{u}+a_{12}\int^{\underline h(t)}_{-\underline h(t)}K(x-y)\underline{v}(y,t)dy
	\geq \widetilde C\lf(\underline u+\underline v\rr)
\end{align}
and
\begin{align}\label{3.19}
	d_2\int^{\underline h(t)}_{-\underline h(t)}J_2(x-y)\underline{v}(y,t)dy-d_2\underline{v}
	-a_{22}\underline{v}+G(\underline{u})
	\geq\widetilde C\lf(\underline v+\int^{\underline h(t)}_{-\underline h(t)}J_2(x-y)\underline{v}(y,t)dy\rr).
\end{align}

The proof is similar to that of \eqref{lower2.6} and \eqref{lower2.7} in the proof of Lemma \ref{lemma1-2}.  First by  \eqref{linear} we have
\begin{equation*}
	\begin{cases}
		-a_{11}\underline{u}(x,t)+a_{12}\underline{v}(x,t)\geq  \rho (\underline{u}+\underline{v}), &|x|\leq \psi_2(t),\\
		-a_{22}\underline{v}(x,t)+G(\underline{u})\geq  \rho (\underline{u}+\underline{v}),&|x|\leq \psi_2(t)
	\end{cases}
\end{equation*}
for some constant $\rho>0$. For $ \psi_2(t)\leq |x|\leq \underline h(t)$, we have
\[
\underline u\leq s\delta_1,\ \underline v=s\delta_2 \mbox{ with } s:=\min\Big\{1,\Psi_2(t,x)\lf(\frac{\underline{h}(t)-|x|}{(t+\sigma_2)^{\beta}}\rr)^{\lambda-1}\Big\}.
\]
Using this and  \eqref{linear}  we easily obtain, as before,
\begin{align*}
	-a_{11}\underline{u}(x,t)+a_{12}\underline{v}(x,t)\geq  \rho (\underline{u}+\underline{v}) \mbox{ for }  \psi_2(t)\leq |x|\leq \underline h(t).
\end{align*}

We next estimate the left hand side of  \eqref{3.18} and \eqref{3.19} by using Example \ref{example2} (c) with
\begin{align*}
	L(t)=\underline h(t), \ \  \tilde \eta_1(t)=\frac{\underline h(t)}{(t+\sigma_2)^{\beta}},\ \  \tilde\eta_2(t)=\frac{\psi_2(t)}{\underline h(t)}\ {\rm or}\ 1,
\end{align*}
Clearly
 $\underline u(x,t)$ meets the conditions for $\phi$ there with $\eta_2(t)=1$, and  $\underline v(x,t)$ meets the conditions for $\phi$ there with $\eta_2(t)=\frac{\psi_2(t)}{\underline h(t)}$. Since
\begin{align*}
	\frac{L(t)}{\tilde\eta_1(t)}=(t+\sigma_2)^{\beta}\to \infty \mbox{ as }  \sigma_2\to \infty,
\end{align*}
 we can apply Example \ref{example2} (c) to conclude that,   for any given small $\epsilon>0$, there is large $h_*=h_*(\epsilon)$ such that for $\sigma_2\geq h_*$, all $|x|\leq \underline h(t)$ and $t>0$, we have
\begin{align*}
	&\int^{\underline h(t)}_{-\underline h(t)}J_1(x-y)\underline{u}(y,t)dy\geq (1-\epsilon)\underline u(x,t),\\
	&\int^{\underline h(t)}_{-\underline h(t)}J(x-y)\underline{v}(y,t)dy\geq (1-\epsilon)\underline v(x,t) \mbox{ for } J\in\{J_2, K\},
\end{align*}
which imply  \eqref{3.18} and \eqref{3.19} in the same way as in the proof of Lemma \ref{lemma1-2}.
 This completes the proof of Step 3.

\medskip

Finally, we apply the comparison principle as in the proof of Lemma \ref{lemma1-2} to complete the proof. It only remains to check that the inequalities required for the initial functions and boundary conditions are satisfied. But these can be done in exactly the same way as in the proof of Lemma \ref{lemma1-2}.
\end{proof}

\subsection{Upper bound for $h(t)$ and $-g(t)$}
\begin{lemma}\label{lemma-upper}
	Suppose  $(\mathbf{J})$ holds and $\alpha\in (1, 2]$ in \eqref{J_1J_2}.    If spreading occurs to \eqref{1}, then there exists  $\tilde D:=\tilde D(\alpha)>0$ such that for all large $t>0$,
	\begin{equation}\label{hg-upper1}
			-g(t),\ h(t)\leq \begin{cases}\tilde D\, t^{\frac{1}{\alpha-1}}&  {\rm if}\ \alpha\in(1,2),\\
			 \tilde D\, t\ln t&   {\rm if}\ \alpha=2.
		\end{cases}
	\end{equation}
\end{lemma}
\begin{proof}
	We first consider the case $\alpha\in(1,2)$. Define
	\begin{equation*}
		\begin{aligned}
			&\bar{h}(t):=(Ct+\sigma)^{\frac{1}{\alpha-1}},~~\bar{g}(t)=-\bar{h}(t),\ \ \ \ \ \ \ \ \ \ t\geq 0,\\
			&\bar{u}(x,t):=M u^*,\ \ x\in[-\bar{h}(t),\bar{h}(t)],\ t\geq 0,\\
			&\bar{v}(x,t):=Mv^*,  \ \  x\in[-\bar{h}(t),\bar{h}(t)],\ t\geq 0,
		\end{aligned}
	\end{equation*}
	where $C$ and $\sigma$ are positive constants to be determined, and $M>1$ is such that $Mu^*>\|u_0\|_\infty,\ Mv^*>\|v_0\|_\infty$. Clearly, by the conditions on the function $G$, we have
	\begin{equation}\label{baruv}
		-a_{11}\bar{u}+a_{12}\bar{v}\leq 0 \mbox{ and }-a_{22}\bar{v}+G(\bar u)\leq 0.
	\end{equation}

	To apply the comparison principle, we first check
	\begin{equation}\label{upper-ac3}
		\begin{cases}
		\dd	\bar{h}'(t)\geq\mu \int_{-\bar{h}(t)}^{\bar{h}(t)}\int^{+\infty}_{\bar{h}(t)}\Big(J_1(x-y)\bar{u}(x,t)+\rho J_2(x-y)\bar{v}(x,t)\Big)dydx,\\[3mm]
			\dd -\bar{h}'(t)\leq-\mu \int_{-\bar{h}(t)}^{\bar{h}(t)}\int_{-\infty}^{-\bar{h}(t)}\Big(J_1(x-y)\bar{u}(x,t)+\rho J_2(x-y)\bar{v}(x,t)\Big)dydx.\\
		\end{cases}
	\end{equation}
	Choosing $\sigma>0$ large enough such that $\bar{h}(t)>1$, together with \eqref{J_1J_2}, we deduce
	\begin{equation}\begin{aligned}\label{3.24}
		&\int_{-\bar{h}(t)}^{\bar{h}(t)}\int^{+\infty}_{\bar{h}(t)}\Big(J_1(x-y)\bar{u}(x,t)+\rho J_2(x-y)\bar{v}(x,t)\Big) dydx\\
		\leq&\max\{\bar u, \bar v\}\int_{-2\bar{h}(t)}^{0}\int^{+\infty}_{0}[J_1(x-y)+\rho J_2(x-y)]dydx\\
		=&\max\{\bar u, \bar v\}\int^{2\bar{h}(t)}_{0}\int^{+\infty}_{x}[J_1(y)+\rho J_2(y)]dydx\\
		\leq& (K_3+\rho )\max\{\bar u, \bar v\}\int^{2\bar{h}(t)}_{0}\int^{+\infty}_{x} J_2(y) dydx\\
		=& (K_3+\rho )\max\{\bar u, \bar v\}\Big[\int^{2\bar{h}(t)}_0y J_2(y)dy+\int^{+\infty}_{2\bar{h}(t)}
		2\bar{h}(t)J_2(y)dy\Big]\\
		\leq & (K_3+\rho )\max\{\bar u, \bar v\} \Big[\int^{1}_0K_2dy +\int^{2\bar{h}(t)}_1K_2 y^{1-\alpha}dy
		+2\bar{h}(t)\int^{+\infty}_{2\bar{h}(t)}\frac{K_2}{y^{\alpha}}dy\Big]\\
		\leq & (K_3+\rho )\max\{\bar u, \bar v\}K_2\Big[ 1+\frac{1}{2-\alpha}\big[(2\bar{h}(t))^{2-\alpha}-1\big]
		+\frac{(2\bar{h}(t))^{2-\alpha}}{\alpha-1}\Big].
	\end{aligned}
	\end{equation}
Hence for $C\gg 1$ and $\sigma\gg1$,
	\begin{align*}
		&\mu \int_{-\bar{h}(t)}^{\bar{h}(t)}\int^{+\infty}_{\bar{h}(t)}\Big(J_1(x-y)\bar{u}(x,t)+\rho J_2(x-y)\bar{v}(x,t)\Big)dydx\\
		\leq &\ \mu  (K_3+\rho )\max\{\bar u, \bar v\} K_2 \Big[1+\frac{2^{2-\alpha}}{(2-\alpha)(\alpha-1)}(Ct+\sigma)^{\frac{2-\alpha}{\alpha-1}}\Big]\\
		\leq&\  \frac{C}{\alpha-1}(Ct+\sigma)^{\frac{2-\alpha}{\alpha-1}}=\bar{h}'(t).
	\end{align*}
	Due to $J_i(x)=J_i(-x)$ and
	\[
	\bar{u}(x,t)=\bar{u}(-x,t),~~\bar{v}(x,t)=\bar{v}(-x,t) \mbox{ for }x\in[-\bar{h}(t),\bar{h}(t)],
	\]
	the second inequality of \eqref{upper-ac3} also holds.
	
	Since both $\bar u$ and $\bar v$ are constant,  by \eqref{baruv}, we obtain
	\begin{align*}
		& d_1\int^{\bar{h}(t)}_{-\bar{h}(t)}J_1(x-y)\bar u(y,t)dy-d_1\bar u-a_{11}\bar u+
		a_{12}\int^{\bar{h}(t)}_{-\bar{h}(t)}K(x-y)\bar v(y,t)dy\\
		 \leq &-a_{11}\bar u+a_{12}\bar v\leq 0=\bar u_t,\\
		&d_2\int^{\bar{h}(t)}_{-\bar{h}(t)}J_2(x-y)\bar v(y,t)dy-d_2\bar v-a_{22}\bar v+G(\bar u)\leq -a_{22}\bar v+G(\bar u)\leq 0=\bar v_t.
	\end{align*}

It remains to check the initial functions. 	Let $\sigma>0$ be large enough such that $\bar{h}(0)=\sigma^{\frac{1}{\alpha-1}}\geq h_0$. 	By the definitions of $\bar u$ and $\bar v$, clearly
\begin{align*}
	&\bar u(x,0)\geq u(x,0) \mbox{ and }\bar v(x,0)\geq v(x,0),\ \ \  x\in[-h_0,h_0],\\
		&\bar u(x,t)>0,~\bar v(x,t)> 0,\ \ \ \ \ \ \ \ \ \ \ \ \ \ \ \ \ \ \ \ \ \  x=\pm \bar h(t),\ t\geq 0.
\end{align*}
Therefore we can apply the comparison principle (see e.g., Lemma 2.3 in \cite{WangDu-JDE}) to conclude that
	\begin{equation*}
		[g(t), h(t)]\subset[-\bar h(t),\bar h(t)]\ \  {\rm and}\ \ \bar u\geq u,\ \bar v\geq  v \mbox{ for }t\geq0,\ x\in[g(t),h(t)].
	\end{equation*}
	which yields \eqref{hg-upper1} for $\alpha\in(1,2)$.
	
	For $\alpha=2$, we define
	\begin{equation*}
		\begin{aligned}
			&\bar{h}(t)=(Ct+\sigma)\ln(Ct+\sigma),~~\bar{g}(t)=-\bar{h}(t),\\
		\end{aligned}
	\end{equation*}
	and define $\bar u$ and $\bar v$ the same as in the case $\alpha\in(1,2)$.
	Now,  in the estimates of \eqref{3.24},  we use $\alpha=2$ and obtain
	\begin{align*}
		&\mu \int_{-\bar{h}(t)}^{\bar{h}(t)}\int^{+\infty}_{\bar{h}(t)}\Big(J_1(x-y)\bar{u}(x,t)+\rho J_2(x-y)\bar{v}(x,t)\Big)dydx\\
		\leq &(K_3+\rho )\mu \max\{\bar u, \bar v\}[ 2K_2+K_2\ln (2\bar h(t))]\\
		= &(K_3+\rho )\mu \max\{\bar u, \bar v\}( 2K_2+K_2\ln [2(Ct+\sigma)\ln(Ct+\sigma)])\\
		=&(K_3+\rho )\mu \max\{\bar u, \bar v\}( 2K_2+K_2\ln 2+K_2\ln (Ct+\sigma)+K_2\ln\ln(Ct+\sigma))\\
		\leq &C\ln(Ct+\sigma)+C=\bar{h}'(t)
	\end{align*}
	provided that  $C$ and $\sigma$ are sufficiently large.
	This implies that \eqref{upper-ac3} holds for $\alpha=2$.
	The rest of the proof is exactly the same as in the case of $\alpha\in(1,2)$, and we conclude that
	\eqref{hg-upper1} holds  when $\alpha=2$. The proof is finished.	
\end{proof}

\subsection{The case that $J_1$ is dominating}
In this subsection, we consider the remaining case that
  $J_1$ is the dominating kernel function.
Therefore, there exist three positive constants $\tilde{K}_1$, $\tilde{K}_2$ and $\tilde{K}_3$ such that
\begin{equation}\label{J_1dom}
\tilde{K}_1\leq \max\{1, |x|^\alpha\} J_1(x) \leq \tilde{K}_2 \ {\rm and}\  \frac{J_2(x)}{J_1(x)}\leq \tilde{K}_3 \mbox{ for }  x\in \R.
\end{equation}
We will show that the estimates in the previous subsections remain valid. The arguments are based on ideas  in the previous subsections, although not completely parallel. 

\begin{lemma}\label{J1-upper1}
Suppose that {\bf (J)} is satisfied,  \eqref{J_1dom} holds with $\alpha\in(1,2)$, and spreading happens to \eqref{1}. Then there exists $D:=D(\alpha)$ such that
\[
-g(t),~h(t)\geq D\, t^{\frac{1}{\alpha-1}} \mbox{ for all large }t>0.
\]

\end{lemma}
\begin{proof}
Similar to the proof of Lemma \ref{lemma1-2},  we construct a suitable lower solution to \eqref{1}.
Let $\lambda\geq 2$ be a constant and define
	\begin{equation*}
		\begin{cases}
			\underline{h}(t):=(C_1t+\sigma_1)^{\frac{1}{\alpha-1}},~~\underline{g}(t):=-\underline{h}(t) &\mbox{ for }t\geq0,\\
		\underline{u}(x,t):=\Psi_1(x,t)\lf(\frac{\underline h(t)-|x|}{\underline h(t)}\rr)^{\lambda-1}\delta_1 &\mbox{ for } t\geq 0,\ |x|\leq \underline{h}(t),\\
	\underline{v}(x,t):=	\lf(\frac{\underline h(t)-|x|}{\underline h(t)}\rr)^{\lambda}\delta_2
&\mbox{ for }  t\geq 0,\ |x|\leq \underline{h}(t),
\end{cases}
\end{equation*}
			where
\begin{equation*}
	\Psi_1(x,t):=
	\begin{cases}
	\dd	\frac{\underline h(t)-|x|}{\underline h(t)},&|x|\leq \psi_1(t),\\[3mm]
	\dd	\frac{\underline h(t)-\psi_1(t)}{\underline h(t)},& |x|\geq \psi_1(t),
	\end{cases}\ \ \  \ \psi_1(t):=\dd\frac{\underline h(t)}{C_2\underline h^{1-\alpha}(t)+1},~t>0,
\end{equation*}
the constant pair $(\delta_1,\delta_2)$ satisfies \eqref{linear} and \eqref{linear1}, and the positive constants $C_1$, $C_2$  and $\sigma_1$ will be chosen later.

We first check the desired inequalities for $\underline{h}'(t)$ and $-\underline{h}'(t)$.
Due to \eqref{J_1dom}, direct calculations give
 \begin{equation*}\begin{aligned}
&\mu\int_{-\underline{h}(t)}^{\underline{h}(t)}\int^{+\infty}_{\underline{h}(t)}\Big(J_1(x-y)\underline{u}(x,t)+\rho J_2(x-y)\underline{v}(x,t)\Big)\rd y \rd x\\
 \geq&\ \mu\int_{0}^{\underline{h}(t)}\int^{+\infty}_{\underline{h}(t)}J_1(x-y)\underline{u}(x,t)\rd y \rd x\\
 \geq &\ \mu\int_{0}^{\underline{h}(t)}\int^{+\infty}_{\underline{h}(t)}J_1(x-y)\lf(\frac{\underline h(t)-x}{\underline h(t)}\rr)^{\lambda}\delta_1\rd y \rd x\\
=&\ \frac{\mu}{\underline{h}^{\lambda}(t)}\int^{0}_{-\underline{h}(t)}\int^{+\infty}_{0}J_1(x-y)(-x)^{\lambda}\delta_1\rd y \rd x\\
=&\ \frac{\mu \delta_1}{\underline{h}^{\lambda}(t)}\int_{0}^{\underline{h}(t)}\int^{+\infty}_{x}J_1(y)x^{\lambda}\rd y \rd x
=\frac{\mu \delta_1}{\underline{h}^{\lambda}(t)}\Big[\int_{0}^{\underline{h}(t)}\int^{y}_{0}+
\int^{+\infty}_{\underline{h}(t)}\int^{\underline{h}(t)}_{0}\Big]J_1(y)x^{\lambda}\rd x \rd y\\
\geq&\ \frac{\mu \delta_1}{\underline{h}^{\lambda}(t)}\int_{0}^{\underline{h}(t)}\int^{y}_{0}J_1(y)x^{\lambda}\rd x \rd y= \frac{\mu \delta_1}{(\lambda+1)\underline{h}^{\lambda}(t)}\int_{0}^{\underline{h}(t)}J_1(y)y^{\lambda+1}\rd y\\
\geq&\ \frac{\mu \delta_1 \tilde{K}_1}{(\lambda+1)\underline{h}^{\lambda}(t)}\int_{1}^{\underline{h}(t)}y^{\lambda+1-\alpha}\rd y
=\frac{\mu \delta_1 \tilde{K}_1(\underline{h}^{\lambda+2-\alpha}(t)-1)}{(\lambda+1)(\lambda+2-\alpha)\underline{h}^{\lambda}(t)}\\
\geq &\  \frac{\mu \delta_1 \tilde{K}_1\underline{h}^{2-\alpha}(t)}{2(\lambda+1)(\lambda+2-\alpha)}
=\frac{\mu \delta_1 \tilde{K}_1}{2(\lambda+1)(\lambda+2-\alpha)}(C_1t+\sigma_1)^{\frac{2-\alpha}{\alpha-1}}\\
\geq&\  \frac{C_1 }{\alpha-1}(C_1t+\sigma_1)^{\frac{2-\alpha}{\alpha-1}}=\underline{h}'(t),
 \end{aligned}\end{equation*}
provided that $\sigma_1\gg 1$ and
\[
C_1\leq \frac{(\alpha-1)\mu \delta_1 \tilde{K}_1}{2(\lambda+1)(\lambda+2-\alpha)}.
\]
The same argument also yields the desired inequality for $-\underline{h}'(t)$.

Next, we verify the inequalities described in \eqref{lower2.4} for $t>0$ and $|x|\in[0,\underline{h}(t)]\setminus \{\psi_1(t)\}$. 
To achieve this, the idea is to show that, with $\sigma_1\gg 1$, for any small $\epsilon>0$,
\begin{equation}\label{lower-inequal}
 \begin{cases}
\underline{u}_t(x,t)&\displaystyle \leq \epsilon \lf(\underline{u}+\int_{-\underline{h}(t)}^{\underline{h}(t)}J_1(x-y)\underline{u}(y,t)dy\rr)\\
&\displaystyle \leq d_1\int^{\underline h(t)}_{-\underline h(t)}J_1(x-y)\underline{u}(y,t)dy-d_1\underline{u}
	-a_{11}\underline{u}+a_{12}\int^{\underline h(t)}_{-\underline h(t)}K(x-y)\underline{v}(y,t)dy;\\
\underline{v}_t(x,t)& \displaystyle\leq \epsilon(\underline{v}+\underline{u})\leq d_2\int^{\underline h(t)}_{-\underline h(t)}J_2(x-y)\underline{v}(y,t)dy-d_2\underline{v}
	-a_{22}\underline{v}+G(\underline{u}),\\
 \end{cases}\end{equation}
where $|x|\in[0,\underline{h}(t)]\setminus \{\psi_1(t)\}$ and $t>0$.

We start by showing  that
for any fixed $\epsilon>0$, there exist suitable constants $C_1$ and $C_2$ such that for $\sigma_1\gg1$, $t\geq0$ and $|x|\in[0,\underline{h}(t)]\setminus\{\psi_1(t)\}$,
\begin{equation*}
\underline{u}_t(x,t)\leq \epsilon \lf(\underline{u}+\int_{-\underline{h}(t)}^{\underline{h}(t)}J_1(x-y)\underline{u}(y,t)dy\rr)
 \mbox{ and }\underline{v}_t(x,t)\leq \epsilon(\underline{v}+\underline{u}).
\end{equation*}

For $x\in[-\underline{h}(t),\underline{h}(t)]$ and $t>0$, we have
\begin{equation*}
 \begin{aligned}
\underline{v}_t(x,t)=\delta_2\lambda\lf(\frac{\underline h(t)-|x|}{\underline h(t)}\rr)^{\lambda-1}\frac{\underline{h}'(t)|x|}{\underline{h}^2(t)}=\frac{C_1\delta_2\lambda}{\alpha-1}\lf(\frac{\underline h(t)-|x|}{\underline h(t)}\rr)^{\lambda-1}\frac{|x|}{\underline{h}(t)}\underline{h}^{1-\alpha}(t).
 \end{aligned}\end{equation*}
We also have that for $ |x|< \psi_1(t)$,
\begin{equation*}
 \begin{aligned}
\underline{u}_t(x,t)=\frac{C_1\delta_1\lambda}{\alpha-1}\lf(\frac{\underline h(t)-|x|}{\underline h(t)}\rr)^{\lambda-1}\frac{|x|}{\underline{h}(t)}\underline{h}^{1-\alpha}(t);
 \end{aligned}\end{equation*}
while for
$\psi_1(t)\leq |x|\leq \underline{h}(t)$,
\begin{equation*}
 \begin{aligned}
\underline{u}_t(x,t)
	=\delta_1(\Psi_1)_t\lf(\frac{\underline h-|x|}{\underline h}\rr)^{\lambda-1}+\Psi_1\frac{C_1\delta_1(\lambda-1)}{\alpha-1}\lf(\frac{\underline h-|x|}{\underline h}\rr)^{\lambda-2}\frac{|x|}{\underline{h}}\underline{h}^{1-\alpha}.\\
 \end{aligned}\end{equation*}
For $\psi_1(t)\leq |x|\leq \underline{h}(t)$, since
\[
(\Psi_1)_t=\frac{\psi_1\underline{h}'-\psi'_1\underline{h}}{\underline{h}^2}
=\frac{C_2(1-\alpha)\underline{h}^{-\alpha}\underline{h}'}{(C_2\underline h^{1-\alpha}(t)+1)^2}<0 \mbox{ due to }\alpha\in(1,2) \mbox{ and }\underline h '>0,
\]
and
\[
\Psi_1=\frac{C_2\underline h^{1-\alpha}(t)}{C_2\underline h^{1-\alpha}(t)+1}\leq C_2 \underline{h}^{1-\alpha}(t),
\]
it follows that
\begin{equation}\label{primeu}
 \begin{aligned}
\underline{u}_t(x,t)\leq \frac{C_1C_2\delta_1(\lambda-1)}{\alpha-1}\lf(\frac{\underline h-|x|}{\underline h}\rr)^{\lambda-2}\frac{|x|}{\underline{h}}\underline{h}^{2-2\alpha} \mbox{ for }\psi_1(t)\leq |x|\leq \underline{h}(t).
 \end{aligned}\end{equation}

Thus, for $|x|< \psi_1(t)$,
\begin{equation*}
 \begin{aligned}
&\underline{u}_t(x,t)-\epsilon\underline{u}(x,t)\\
=&\ \frac{C_1\delta_1\lambda}{\alpha-1}\lf(\frac{\underline h(t)-|x|}{\underline h(t)}\rr)^{\lambda-1}\frac{|x|}{\underline{h}(t)}\underline{h}^{1-\alpha}(t)-\epsilon\lf(\frac{\underline h(t)-|x|}{\underline h(t)}\rr)^{\lambda}\delta_1\\
=&\ \epsilon\delta_1\lf(\frac{\underline h(t)-|x|}{\underline h(t)}\rr)^{\lambda-1}
\lf( \frac{C_1\lambda}{\epsilon(\alpha-1)}  \frac{|x|\underline{h}^{1-\alpha}(t)}{\underline{h}(t)}+ \frac{|x|}{\underline{h}(t)}-1 \rr)\\
\leq&\ \epsilon\delta_1\lf(\frac{\underline h(t)-|x|}{\underline h(t)}\rr)^{\lambda-1}
\lf[\lf( \frac{C_1\lambda}{\epsilon(\alpha-1)}\underline{h}^{1-\alpha}(t)+1\rr)\frac{1}{C_2\underline{h}^{1-\alpha}(t)+1}-1\rr]\leq0
 \end{aligned}\end{equation*}
if \[
C_2\geq \frac{C_1\lambda}{\epsilon(\alpha-1)}.
\]
Similarly, we see that
\[
\underline{v}_t(x,t)-\epsilon\underline{v}(x,t)\leq 0 \mbox{ for }|x|\leq\psi_1(t),~~t>0.
\]

For $|x|\in[\psi_1(t),\underline{h}(t)]$ and $t>0$,  our earlier calculations yield
\begin{equation*}
 \begin{aligned}
\underline{v}_t(x,t)-\epsilon\underline{u}(x,t)
=&\ \frac{C_1\delta_2\lambda}{\alpha-1}\lf(\frac{\underline h(t)-|x|}{\underline h(t)}\rr)^{\lambda-1}\underline{h}^{1-\alpha}(t)-\epsilon\frac{\delta_1C_2\underline{h}^{1-\alpha}(t)}{C_2\underline{h}^{1-\alpha}(t)+1}\lf(\frac{\underline h(t)-|x|}{\underline h(t)}\rr)^{\lambda-1}\\
\leq&\ \lf(\frac{\underline h(t)-|x|}{\underline h(t)}\rr)^{\lambda-1}\underline{h}^{1-\alpha}(t)\lf[
\frac{C_1\delta_2\lambda}{\alpha-1}-\frac{\epsilon \delta_1C_2}{2}\rr]\leq0
 \end{aligned}\end{equation*}
 provided that
 \[
 C_2\geq\frac{2\delta_2C_1\lambda}{\delta_1\epsilon(\alpha-1)}.
 \]
Moreover, for such $|x|$, it follows from \eqref{primeu} that
\begin{equation*}
 \begin{aligned}
\underline{u}_t(x,t)\leq \frac{C_1C_2^{\lambda-1}
\delta_1(\lambda-1)}{\alpha-1}\frac{\underline{h}^{\lambda(1-\alpha)}(t)}
{(C_2\underline{h}^{1-\alpha}(t)+1)^{\lambda-2}}
\leq\frac{C_1C_2^{\lambda-1}
\delta_1(\lambda-1)}{\alpha-1}\underline{h}^{\lambda(1-\alpha)}(t).
 \end{aligned}\end{equation*}
And by \eqref{J_1dom}, for $x\in[\psi_1(t),\underline{h}(t)]$ and $t>0$, we have
\begin{equation}\label{J-1bound}
 \begin{aligned}
&\int^{\underline{h}(t)}_{-\underline{h}(t)}J_1(x-y)\underline{u}(y,t)\rd y=
\int^{\underline{h}(t)-x}_{-\underline{h}(t)-x}J_1(y)\underline{u}(x+y,t)\rd y\\
\geq&\int^{-\underline{h}(t)/4}_{-\underline{h}(t)/2}J_1(y)\underline{u}(x+y,t)\rd y
\geq\tilde{K}_1\delta_1\int^{-\underline{h}(t)/4}_{-\underline{h}(t)/2}|y|^{-\alpha}\lf(\frac{\underline h-|x+y|}{\underline h(t)}\rr)^{\lambda}\rd y\\
\geq &\ \frac{\tilde{K}_1\delta_1}{4^{\lambda}}\int^{-\underline{h}(t)/4}_{-\underline{h}(t)/2}|y|^{-\alpha}\rd y
=\frac{\tilde{K}_1\delta_1}{4^{\lambda}}\int^{\underline{h}(t)/2}_{\underline{h}(t)/4}y^{-\alpha}\rd y\\
=&\ \frac{\tilde{K}_1\delta_1(4^{\alpha-1}-2^{\alpha-1})}{(\alpha-1)4^{\lambda}}\underline{h}^{1-\alpha}(t).
 \end{aligned}\end{equation}
 Similarly, \eqref{J-1bound} holds for $x\in[-\underline{h}(t),-\psi_1(t)]$ since $J_1(x)$ and $\underline u(x,t)$ are symmetric in $x$ over $[-\underline{h}(t),\underline{h}(t)]$.
 Therefore, in view of $\lambda (1-\alpha)<1-\alpha<0$ and $\underline h(t)\geq \sigma_1^{1/(\alpha-1)}$, we have, for any $\epsilon>0$ and all $\sigma_1\gg1$,
 \begin{equation*}
 \begin{aligned}
\underline{u}_t(x,t)\leq \epsilon \int^{\underline{h}(t)}_{-\underline{h}(t)}J_1(x-y)\underline{u}(y,t)\rd y \ \forall t>0,\ x\in[-\underline{h}(t),-\psi_1(t)]\cup [\psi_1(t), \underline h(t)].
 \end{aligned}\end{equation*}

Now to complete the proof of \eqref{lower-inequal}, it remains to show, for $|x|\leq\underline{h}(t)$ and $t>0$,
\begin{equation}\begin{cases}\label{2-lower2.6}
		\displaystyle d_1\int^{\underline h(t)}_{-\underline h(t)}J_1(x-y)\underline{u}(y,t)dy-d_1\underline{u}
	-a_{11}\underline{u}+a_{12}\int^{\underline h(t)}_{-\underline h(t)}K(x-y)\underline{v}(y,t)dy\\[2mm]
	\hspace{8cm} \displaystyle\geq  \epsilon\lf(\underline u+\int^{\underline h(t)}_{-\underline h(t)}J_1(x-y)\underline{u}(y,t)dy\rr);\\
		\displaystyle d_2\int^{\underline h(t)}_{-\underline h(t)}J_2(x-y)\underline{v}(y,t)dy-d_2\underline{v}
	-a_{22}\underline{v}+G(\underline{u})
	\geq \epsilon\lf(\underline v+\underline u\rr). \end{cases}
\end{equation}

By \eqref{linear} and \eqref{linear1}, for $|x|\leq\psi_1(t)$ and $t>0$, we have
\begin{equation*}\begin{aligned}
-a_{11}\underline{u}+a_{12}\underline{v}\geq \rho (\underline{u}+\underline{v})\mbox{ and }
-a_{22}\underline{v}+G(\underline{u})\geq \rho (\underline{u}+\underline{v}).
\end{aligned}
\end{equation*}
When $|x|\in[\psi_1(t),\underline{h}(t)]$, by \eqref{linear1} and the fact that
\[
\underline u=s \delta_1,\ \  \underline v\leq s \delta_2
\]
 with $s:= \Psi_1(x,t)\lf(\frac{\underline h(t)-|x|}{\underline h(t)}\rr)^{\lambda-1}$, we obtain
\begin{equation*}\begin{aligned}
&-a_{22}\underline{v}+G(\underline{u})-\rho(\underline u+\underline v) \geq -a_{22} (s\delta_2)+G(s\delta_1)-\rho(s\delta_1+s\delta_2)\geq 0.
\end{aligned}
\end{equation*}

Next we apply Proposition \eqref{prop2.3}. By Example \eqref{example2} $(a)$ and $(b)$, for any given small $\epsilon>0$, there exists a 
 large $L(\epsilon)>0$ such that for $\underline{h}(t)>L(\epsilon)$ (which is guaranteed when $\sigma_1\gg 1$) and $|x|\leq \underline{h}(t)$, 
\begin{align*}
	&\int^{\underline h(t)}_{-\underline h(t)}J_1(x-y)\underline{u}(y,t)dy\geq (1-\epsilon)\underline u(x,t),\\
	&\int^{\underline h(t)}_{-\underline h(t)}J(x-y)\underline{v}(y,t)dy\geq (1-\epsilon)\underline v(x,t) \mbox{ for } J\in\{J_2, K\}.
\end{align*}
Therefore, for $|x|\leq \underline{h}(t)$ and $t>0$, 
\begin{equation*}\begin{aligned}
		 &d_2\int^{\underline h(t)}_{-\underline h(t)}J_2(x-y)\underline{v}(y,t)dy-d_2\underline{v}
	-a_{22}\underline{v}+G(\underline{u}) \\
	&\geq-d_2\epsilon \underline{v}+\rho(\underline{u}+\underline{v})\geq \frac{\rho}{2}(\underline{u}+\underline{v})\geq \epsilon(\underline{u}+\underline{v})
\end{aligned}
\end{equation*}
provided that $\epsilon>0$ is small enough such that $\rho>2d_2\epsilon$. This proves the second inequality in \eqref{2-lower2.6}.

We now prove the first inequality in \eqref{2-lower2.6}. For $|x|\leq\psi_1(t)$ and $t>0$,
\begin{equation*}\begin{aligned}
		& (d_1-\epsilon)\int^{\underline h(t)}_{-\underline h(t)}J_1(x-y)\underline{u}(y,t)dy-d_1\underline{u}
	-a_{11}\underline{u}+a_{12}\int^{\underline h(t)}_{-\underline h(t)}K(x-y)\underline{v}(y,t)dy,\\[2mm]
	&\geq  ( d_1-\epsilon)(1-\epsilon)\underline{u}-d_1\underline{u}+\rho(\underline{u}+\underline{v})-\epsilon a_{12}\underline{v}\geq \frac{\rho}{2}(\underline{u}+\underline{v})\geq  \epsilon \underline{u}
\end{aligned}
\end{equation*}
provided that $\rho\geq 2(d_1+1+a_{12})\epsilon$. 

For $|x|\in[\psi_1(t),\underline{h}(t)]$, we have, for any given $M>0$ and all $\sigma_1\gg 1$,
\begin{align*}
	\int^{\underline h(t)}_{-\underline h(t)}J_1(x-y)\underline{u}(y,t)dy\geq M \underline u(x,t) \mbox{ for }  t\geq 0,
\end{align*}
which is a consequence of \eqref{J-1bound} and the fact that
\begin{align*}
	\underline u =&\  \delta_1\frac{\underline h-\psi_1}{\underline h}\lf(\frac{\underline h-|x|}{\underline h}\rr)^{\lambda-1}
	\leq \delta_1	\lf(\frac{\underline h-\psi_1}{\underline h}\rr)^{\lambda}= \delta_1\lf(\frac{C_2\underline h^{1-\alpha}}{C_2\underline h^{1-\alpha}+1}\rr)^\lambda
	\leq \delta_1 C_2^{\lambda}\underline h^{\lambda(1-\alpha)}.
\end{align*}
Therefore, for $|x|\in [\psi_1(t), \underline h(t)]$ and $t>0$,
\begin{equation*}\begin{aligned}
		& (d_1-\epsilon)\int^{\underline h(t)}_{-\underline h(t)}J_1(x-y)\underline{u}(y,t)dy-d_1\underline{u}
	-a_{11}\underline{u}+a_{12}\int^{\underline h(t)}_{-\underline h(t)}K(x-y)\underline{v}(y,t)dy,\\[2mm]
	&\geq  ( d_1-\epsilon)M{\underline{u}}-d_1\underline{u}-a_{11}\underline u\geq   \epsilon \underline{u}
\end{aligned}
\end{equation*}
provided that $M>0$ is large enough such that $( d_1-\epsilon)M >d_1+a_{11}+\epsilon$. 
We have thus proved that \eqref{2-lower2.6}, and hence \eqref{lower-inequal}, holds for all $\sigma_1\gg1$ and $1\gg \epsilon>0$.

Finally, we show that the boundary and initial conditions also meet the requirement of the comparison principle.
Since spreading happens, we have
\[\lim_{t\rightarrow \infty}[-g(t)]=\lim_{t\to\infty} h(t)=\infty\]
and
\[
\lim_{t\rightarrow \infty}(u(x,t),v(x,t)) =(u^*,v^*) \mbox{ locally uniformly in }\R.
\]
Therefore, we can find some large $t_0>0$ such that,
\[
[-\underline{h}(0),\underline{h}(0)]\subset[g(t_0),h(t_0)],
\]
and for $x\in[-\underline{h}(0),\underline{h}(0)]$
\[
u(x,t_0)\geq\delta_1\geq\underline{u}(x,0)\mbox{ and }v(x,t_0)\geq\delta_2>\underline{v}(x,0)
\]
with our earlier choice of $(\delta_1,\delta_2)$.
In addition, we have $(\underline{u}(x,t),\underline{v}(x,t))=(0,0)$ for $x=\pm \underline{h}(t)$ and $t\geq0$. Therefore, we can apply the comparison principle (as in the proof of Lemma \ref{lemma1-2}) to obtain that
\[
-\underline h(t)\geq g(t_0+t),
		 \ \ \underline{h}(t)\leq h(t_0+t) \mbox{ for }t\geq0,
\]
which clearly implies the desired estimates for $g(t)$ and $h(t)$.
\end{proof}

\begin{lemma}\label{JIDOM2-LOWER}

Suppose that $(\mathbf{J})$ holds. If spreading occurs and \eqref{J_1dom} holds with $\alpha=2$, then there exists $D>0$ such that
\begin{equation*}
-g(t),h(t)\geq D\;t\ln t\mbox{ for all large } t>0.\\
\end{equation*}
\end{lemma}

\begin{proof}
Similar to the proof of Lemma \ref{alpha2}, we will prove the desired inequalities by constructing an appropriate lower solution to problem \eqref{1}.
Fix constants $\beta\in(0,\frac 12)$ and $\lambda> \frac{1}{\beta}>2$. Then define
\begin{equation*}
\begin{cases}
\underline{h}(t):=C_1(t+\sigma_2)\ln (t+\sigma_2),& t\geq0,\\
	\underline{u}(x,t):=\min\big\{1,\Psi_2(t,x)\lf(\frac{\underline{h}(t)-|x|}{(t+\sigma_2)^{\beta}}\rr)^{\lambda-1}\big\}
	\delta_1, & x\in[-\underline{h}(t),\underline{h}(t)],~t\geq0,\\
	\underline{v}(x,t):=\min\big\{1,\lf(\frac{\underline{h}(t)-|x|}{(t+\sigma_2)^{\beta}}\rr)^{\lambda}\big\}
	\delta_2, & x\in[-\underline{h}(t),\underline{h}(t)],~t\geq0,
	\end{cases}
\end{equation*}
where
\begin{equation*}
	\Psi_2(x,t):=
	\begin{cases}
		\dd	\frac{\underline h(t)-|x|}{(t+\sigma_2)^{\beta}},&|x|\leq \psi_2(t),\\[3mm]
		\dd	\frac{\underline h(t)-\psi_2(t)}{(t+\sigma_2)^{\beta}},& |x|\geq \psi_2(t),
	\end{cases}\ \ \  \ \psi_2(t):=\dd\frac{\underline h(t)-C_3\ln (t+\sigma_2)}{C_2(t+\sigma_2)^{-1}+1},
\end{equation*}
$(\delta_1,\delta_2)$ is selected by \eqref{linear} and \eqref{linear1},  while $C_1$, $C_2$, $C_3$
and $\sigma_2$ are positive constants to be determined.

To show $(\underline u, \underline v, -\underline h, \underline h)$ forms a lower solution, we first check the inequalities satisfied by $\underline{h}'(t)$ and $-\underline{h}'(t)$ for $t>0$.
 We calculate,  using \eqref{J_1dom} when needed,
 \begin{equation*}\begin{aligned}
& \mu\int_{-\underline{h}(t)}^{\underline{h}(t)}\int^{+\infty}_{\underline{h}(t)}\Big(J_1(x-y)\underline{u}(x,t)+\rho J_2(x-y)\underline{v}(x,t)\Big)\rd y \rd x\\
\geq &\ \mu\int_{-\underline{h}(t)}^{\underline{h}(t)}\int^{+\infty}_{\underline{h}(t)}J_1(x-y)\underline{u}(x,t)\rd y \rd x\geq\mu\delta_1
\int_{0}^{\underline{h}(t)-(t+\sigma_2)^{\beta}}\int^{+\infty}_{\underline{h}(t)}J_1(x-y)\rd y \rd x\\
=&\ \mu\delta_1
\int_{-\underline{h}(t)}^{-(t+\sigma_2)^{\beta}}\int^{+\infty}_{0}J_1(x-y)\rd y \rd x
=\mu\delta_1\int^{\underline{h}(t)}_{(t+\sigma_2)^{\beta}}\int^{+\infty}_{x}J_1(y)\rd y \rd x\\
=&\ \mu\delta_1\lf[
\int^{\underline{h}(t)}_{(t+\sigma_2)^{\beta}}\int^{y}_{(t+\sigma_2)^{\beta}}+
\int_{\underline{h}(t)}^{+\infty}\int^{\underline{h}(t)}_{(t+\sigma_2)^{\beta}}\rr]J_1(y)\rd x\rd y
\geq \mu\delta_1\int^{\underline{h}(t)}_{(t+\sigma_2)^{\beta}}\int^{y}_{(t+\sigma_2)^{\beta}}J_1(y)\rd x\rd y\\
=&\ \mu\delta_1\int^{\underline{h}(t)}_{(t+\sigma_2)^{\beta}}J_1(y)[y-(t+\sigma_2)^{\beta}]\rd y
\geq\mu\delta_1 \tilde{K}_1\int^{\underline{h}(t)}_{(t+\sigma_2)^{\beta}}\frac{y-(t+\sigma_2)^{\beta}}{y^2}\rd y\\
=&\ \mu\delta_1 \tilde{K}_1\Big(\ln (\underline{h}(t))-\beta\ln(t+\sigma_2)+\frac{(t+\sigma_2)^{\beta}}{\underline{h}(t)}-1\Big)\\
\geq&\ \mu\delta_1 \tilde{K}_1\big[(1-\beta)\ln (t+\sigma_2)+1\big]
\geq C_1\big[\ln (t+\sigma_2)+1\big]=\underline{h}'(t),\\
 \end{aligned}\end{equation*}
provided that $\sigma_2\gg1$ and  $C_1\leq \mu\delta_1 \tilde{K}_1(1-\beta)$.
The same calculations can be applied to obtain the desired inequality for $-\underline{h}'(t).$

Now we prove that for $t>0$ and $|x|\in [0, \underline h(t)]\setminus\{ \underline{h}(t)-(t+\sigma_2)^{\beta},\psi_2(t)\}$, 
\begin{equation}\label{lowerJ1-1}
\begin{cases}
	\dd\underline u_t\leq d_1\int^{\underline h(t)}_{-\underline h(t)}J_1(x-y)\underline{u}(y,t)dy-d_1\underline{u}
	-a_{11}\underline{u}+a_{12}\int^{\underline h(t)}_{-\underline h(t)}K(x-y)\underline{v}(y,t)dy,\\[3mm]
	\dd\underline{v}_t\leq d_2\int^{\underline h(t)}_{-\underline h(t)}J_2(x-y)\underline{v}(y,t)dy-d_2\underline{v}
-a_{22}\underline{v}+G(\underline{u}),
\end{cases}
\end{equation}
where $\{\pm[ \underline{h}(t)-(t+\sigma_2)^{\beta}],\pm\psi_2(t)\}$ is the set of discontinuous points of $\underline{u}_t$ and $|x|=\underline{h}(t)-(t+\sigma_2)^{\beta}$ is the discontinuous point of $\underline{v}_t$.

Let us note that, for $\sigma_2\gg1 $, $\underline{h}(t)-(t+\sigma_2)^{\beta}<\psi_2(t)$ for all $t>0$.
Similar to the proof of Lemma \ref{J1-upper1}, we deduce \eqref{lowerJ1-1} via the following inequalities; namely, for some constant $M>0$,
\begin{equation}\label{lower-inequa2}
 \begin{cases}
\underline{u}_t(x,t)\leq M\displaystyle \lf(\underline{u}+\int_{-\underline{h}(t)}^{\underline{h}(t)}J_1(x-y)\underline{u}(y,t)dy\rr)\\
\hspace{1.3cm}\leq \displaystyle d_1\int^{\underline h(t)}_{-\underline h(t)}J_1(x-y)\underline{u}(y,t)dy-d_1\underline{u}
	-a_{11}\underline{u}+a_{12}\int^{\underline h(t)}_{-\underline h(t)}K(x-y)\underline{v}(y,t)dy;\\
\displaystyle\underline{v}_t(x,t)\leq M(\underline{v}+\underline{u})\leq d_2\int^{\underline h(t)}_{-\underline h(t)}J_2(x-y)\underline{v}(y,t)dy-d_2\underline{v}
	-a_{22}\underline{v}+G(\underline{u}),\\
 \end{cases}\end{equation}
where $t>0$ and $|x|\in [0, \underline h(t)]\setminus\{ \underline{h}(t)-(t+\sigma_2)^{\beta},\psi_2(t)\}$.

So we first prove, for such $t$ and $|x|$, and any $M>0$,
\begin{equation}\label{left-inequal-1}
\underline{u}_t(x,t)\leq M \lf(\underline{u}+\int_{-\underline{h}(t)}^{\underline{h}(t)}J_1(x-y)\underline{u}(y,t)dy\rr),~\underline{v}_t(x,t)\leq M(\underline{v}+\underline{u}).
\end{equation}
 
For $0\leq|x|<\underline{h}(t)-(t+\sigma_2)^{\beta}$, clearly
\[
(\underline{u}_t,\underline{v}_t)=(0,0) \mbox{ for }t>0,
\]
and so \eqref{left-inequal-1} holds trivially in this case.

For $\underline{h}(t)-(t+\sigma_2)^{\beta}\leq|x|<\psi_2(t)$, we have
\[
(\underline{u}(x,t),\underline{v}(x,t))=\lf(\frac{\underline{h}(t)-|x|}{(t+\sigma_2)^{\beta}}\rr)^{\lambda}
(\delta_1,\delta_2) \mbox{ for }t>0,
\]
and so,
for $t>0$ and $\sigma_2\gg1$,
\begin{equation*}
 \begin{aligned}
\underline{u}_t(x,t)&=\delta_1\lambda\lf(\frac{\underline{h}(t)-|x|}{(t+\sigma_2)^{\beta}}\rr)^{\lambda-1}\lf[
\frac{C_1(1-\beta)[\ln(t+\sigma_2)+1]}{(t+\sigma_2)^{\beta}}+\frac{\beta|x|}{(t+\sigma_2)^{\beta+1}}
\rr]\\
&\leq \delta_1\lambda\lf(\frac{\underline{h}(t)-|x|}{(t+\sigma_2)^{\beta}}\rr)^{\lambda-1}\lf[
\frac{C_1\ln(t+\sigma_2)}{(t+\sigma_2)^{\beta}}+\frac{|x|}{(t+\sigma_2)^{\beta+1}}
\rr],
 \end{aligned}\end{equation*}
and similarly,
\begin{equation}\label{lowerv}
 \begin{aligned}
\underline{v}_t(x,t)&
\leq \delta_2\lambda\lf(\frac{\underline{h}(t)-|x|}{(t+\sigma_2)^{\beta}}\rr)^{\lambda-1}\lf[
\frac{C_1\ln(t+\sigma_2)}{(t+\sigma_2)^{\beta}}+\frac{|x|}{(t+\sigma_2)^{\beta+1}}
\rr] .\\
 \end{aligned}\end{equation}
Therefore, 
\begin{equation*}
 \begin{aligned}
&\underline{u}_t(x,t)-M \underline{u}(x,t)\\
&\leq\delta_1\lambda\lf(\frac{\underline{h}(t)-|x|}{(t+\sigma_2)^{\beta}}\rr)^{\lambda-1}\lf[
\frac{C_1\ln(t+\sigma_2)}{(t+\sigma_2)^{\beta}}+\frac{|x|}{(t+\sigma_2)^{\beta+1}}\rr]-M\delta_1
\lf(\frac{\underline{h}(t)-|x|}{(t+\sigma_2)^{\beta}}\rr)^{\lambda}\\
&\leq \frac{\delta_1}{(t+\sigma_2)^{\beta}}\lf(\frac{\underline{h}(t)-|x|}{(t+\sigma_2)^{\beta}}\rr)^{\lambda-1}
\lf[2\lambda C_1\ln(t+\sigma_2)-\frac{M(C_1C_2+C_3)\ln(t+\sigma_2)}{C_2(t+\sigma_2)^{-1}+1}\rr]\\
&\leq  \frac{\delta_1}{(t+\sigma_2)^{\beta}}\lf(\frac{\underline{h}(t)-|x|}{(t+\sigma_2)^{\beta}}\rr)^{\lambda-1}
\lf[2\lambda C_1\ln(t+\sigma_2)-\frac{M(C_1C_2+C_3)\ln(t+\sigma_2)}{2}\rr]<0
\end{aligned}\end{equation*}
provided that $\sigma_2\gg1 $ and $C_1, C_2, C_3$ are chosen such that
\[
M>\frac{4\lambda C_1}{C_1C_2+C_3}.
\]
Similarly, with such choice of $\sigma_2$ and $C_1, C_2, C_3$,  we have
\[
\underline{v}_t(x,t)-M \underline{v}(x,t)\leq0 \mbox{ for }\underline{h}(t)-(t+\sigma_2)^{\beta}\leq |x|\leq \psi_2(t),~t>0.
\]

For $\psi_2(t)\leq |x|\leq\underline{h}(t)$, since \eqref{lowerv} still holds in this case, we see that
\begin{equation*}
 \begin{aligned}
&\underline{v}_t(x,t)-M \underline{u}(x,t)\\
&\leq\delta_2\lambda\lf(\frac{\underline{h}(t)-|x|}{(t+\sigma_2)^{\beta}}\rr)^{\lambda-1}\lf[
\frac{C_1\ln(t+\sigma_2)}{(t+\sigma_2)^{\beta}}+\frac{|x|}{(t+\sigma_2)^{\beta+1}}
\rr]-M\Psi_2(t,x)\lf(\frac{\underline{h}(t)-|x|}{(t+\sigma_2)^{\beta}}\rr)^{\lambda-1}
	\delta_1\\
&\leq  \frac{1}{(t+\sigma_2)^{\beta}}\lf(\frac{\underline{h}(t)-|x|}{(t+\sigma_2)^{\beta}}\rr)^{\lambda-1}
\lf[2\lambda\delta_2 C_1\ln(t+\sigma_2)-\frac{\delta_1 M(C_1C_2+C_3)\ln(t+\sigma_2)}{2}\rr]
<0,
\end{aligned}\end{equation*}
provided that $\sigma_2\gg1 $ and
\[
M>\frac{4\lambda \delta_1C_1}{\delta_2(C_1C_2+C_3)}.
\]
We have thus proved  the second inequality in \eqref{left-inequal-1}, and for the first inequality in \eqref{left-inequal-1}, it only remains to prove it for $\psi_2(t)\leq |x|\leq \underline{h}(t)$ and $t>0$.
For $|x|$ in this range, 
since
\begin{equation*}
 \begin{aligned}
 &\Psi_2(x,t)=(C_1C_2+C_3)\frac{\ln(t+\sigma_2)}{[C_2(t+\sigma_2)^{-1}+1](t+\sigma_2)^{\beta}}
 \leq (C_1C_2+C_3)\frac{\ln(t+\sigma_2)}{(t+\sigma_2)^{\beta}},
 \\
&(\Psi_2)_t(x,t)=(C_1C_2+C_3)\frac{\big[1-\beta\ln(t+\sigma_2)\big](t+\sigma_2)
+C_2+(1-\beta)\ln(t+\sigma_2)}{(t+\sigma_2)^{\beta+2}[C_2(t+\sigma_2)^{-1}+1]^2}<0
 \end{aligned}\end{equation*}
 when $\sigma_2\gg1$,
we further have
\begin{equation*}
 \begin{aligned}
\underline{u}_t(x,t)=&\ (\Psi_2)_t\lf(\frac{\underline{h}(t)-|x|}{(t+\sigma_2)^{\beta}}\rr)^{\lambda-1}\delta_1
\\
&\ +\Psi_2(\lambda-1)\lf(\frac{\underline{h}(t)-|x|}{(t+\sigma_2)^{\beta}}\rr)^{\lambda-2}\lf[
\frac{C_1(1-\beta)[\ln(t+\sigma_2)+1]}{(t+\sigma_2)^{\beta}}+\frac{\beta|x|}{(t+\sigma_2)^{\beta+1}}
\rr]\\
\leq&\ (\Psi_2)^{\lambda-1}(\lambda-1)\lf[
\frac{C_1\ln(t+\sigma_2)}{(t+\sigma_2)^{\beta}}+\frac{\underline{h}(t)}{(t+\sigma_2)^{\beta+1}}
\rr]\\
\leq&\ 2\delta_1\lambda C_1\left(C_1C_2+C_3\right)^{\lambda-1}\left(\frac{\ln (t+\sigma_2)}{(t+\sigma_2)^{\beta}}\right)^\lambda.
 \end{aligned}\end{equation*}

We next estimate $\int^{\underline{h}(t)}_{-\underline{h}(t)}J_1(x-y)\underline{u}(y,t)\rd y$.
By the symmetry of $J_1(x)$ and $\underline{u}(x, t)$ for $x\in [-\underline{h}(t),\underline{h}(t)]$, we only consider the case where $x\in[\psi_2(t),\underline{h}(t)]$  since the same argument can be applied to $x\in[-\underline{h}(t),-\psi_2(t)]$.
Due to
\[
\frac{\psi_2(t)}{\underline{h}(t)}=\frac{1-\frac{C_3}{C_1}(t+\sigma_2)^{-1}}{C_2(t+\sigma_2)^{-1}+1}\rightarrow 1 \mbox{ as }\sigma_2\rightarrow \infty \mbox{ uniformly for }t>0,
\]
we have, for $\sigma_2\gg1 $,  
\[
[-\underline{h}(t)/2,-\underline{h}(t)/4]\subset[-\underline{h}(t)-x,\underline{h}(t)-x] \mbox{ for } x\in[\psi_2(t),\underline{h}(t)]
\]
and
\[
x+y\in[\underline{h}(t)/4,3\underline{h}(t)/4] \mbox{ for } x\in[\psi_2(t),\underline{h}(t)] \mbox{ and }y\in[-\underline{h}(t)/2,-\underline{h}(t)/4].
\]
It follows that, for $x\in[\psi_2(t),\underline{h}(t)]$,
\begin{equation}\label{2-J-1bound}
 \begin{aligned}
&\int^{\underline{h}(t)}_{-\underline{h}(t)}J_1(x-y)\underline{u}(y,t)\rd y=
\int^{\underline{h}(t)-x}_{-\underline{h}(t)-x}J_1(y)\underline{u}(x+y,t)\rd y\\
\geq&\int^{-\underline{h}(t)/4}_{-\underline{h}(t)/2}J_1(y)\underline{u}(x+y,t)\rd y
\geq \tilde{K}_1\delta_1\int^{-\underline{h}(t)/4}_{-\underline{h}(t)/2}y^{-2} \rd y=\frac{2\tilde{K}_1\delta_1}{\underline{h}(t)}=\frac{2\tilde{K}_1\delta_1}{C_1(t+\sigma_2)\ln (t+\sigma_2)},
 \end{aligned}\end{equation}
where we have used $\underline{u}(x+y,t)=\delta_1$ for $x+y\leq 3\underline{h}(t)/4<\underline{h}(t)-(t+\sigma_2)^{\beta}$.
Thus we have, for $\psi_2(t)\leq |x|\leq \underline{h}(t)$ and $t>0$,
\begin{equation*}
 \begin{aligned}
&\underline{u}_t(x,t)-M\int^{\underline{h}}_{-\underline{h}}J_1(x-y)\underline{u}(y,t)\rd y\\
\leq
 &\ 2C_1\delta_1\lambda\left(C_1C_2+C_3\right)^{\lambda-1}\left(\frac{\ln (t+\sigma_2)}{(t+\sigma_2)^{\beta}}\right)^{\lambda}-
 \frac{2M\tilde{K}_1\delta_1}{C_1(t+\sigma_2)\ln (t+\sigma_2)}
 <0
 \end{aligned}\end{equation*}
provided that $\lambda\beta>1$ and  $\sigma_2\gg 1$. We have now completed the proof of the first inequality in \eqref{left-inequal-1}, keeping in mind that it is possible to choose $C_1, C_2$ and $C_3$ to meet all the stated requirements.

Next we prove the remaining inequalities in \eqref{lower-inequa2}.
By Proposition \ref{prop2.3} and Example \ref{example2} $(c)$, for any small $\epsilon>0$, there exists a large  $L(\epsilon)>0$ such that
for $\sigma_2\gg L(\epsilon)$, $x\in[-\underline{h}(t),\underline{h}(t)]$ and $t>0$, we have
\begin{align*}
	&\int^{\underline h(t)}_{-\underline h(t)}J_1(x-y)\underline{u}(y,t)dy\geq (1-\epsilon)\underline u(x,t),\\
	&\int^{\underline h(t)}_{-\underline h(t)}J(x-y)\underline{v}(y,t)dy\geq (1-\epsilon)\underline v(x,t) \mbox{ for } J\in\{J_2, K\}.
\end{align*}
Moreover, for $0\leq |x|<\psi_2(t)$ and $t>0$, it follows from \eqref{linear1} that
\[
-a_{11}\underline{u}+a_{12}\underline{v}\geq \rho(\underline{u}+\underline{v} )\mbox{ and }
-a_{22}\underline{v}+G(\underline{u})\geq\rho(\underline{u}+\underline{v} ).
\]
Therefore, for such $|x|$, $t$ and  $\tilde{M}\in (0, \min\{\frac{\rho-2d_1\epsilon}{4},d_1\})$ we have
\begin{equation*}\begin{aligned}
		& (d_1-\tilde{M})\int^{\underline h(t)}_{-\underline h(t)}J_1(x-y)\underline{u}(y,t)dy-d_1\underline{u}
	-a_{11}\underline{u}+a_{12}\int^{\underline h(t)}_{-\underline h(t)}K(x-y)\underline{v}(y,t)dy\\[2mm]
	&\geq  ( d_1-\tilde{M})(1-\epsilon)\underline{u}-d_1\underline{u}-\epsilon a_{12}\underline v+\rho(\underline{u}+\underline{v})\geq \frac{\rho}{2}(\underline{u}+\underline{v})\geq  \tilde{M}\underline{u},
\end{aligned}
\end{equation*}
if $\epsilon>0$ is small enough. This proves the desired inequality for $\underline u$ with $M=\tilde M$ for the case $0\leq |x|<\psi_2(t)$.

Moreover, for such $\tilde M$, $\epsilon$ and $|x|$, 
\begin{equation*}\begin{aligned}
		& d_2\int^{\underline h(t)}_{-\underline h(t)}J_2(x-y)\underline{v}(y,t)dy-d_2\underline{v}
	-a_{22}\underline{v}+G(\underline u) \\[2mm]
	&\geq  -d_2\epsilon\underline{v}+\rho(\underline{u}+\underline{v})\geq \frac{\rho}{2}(\underline{u}+\underline{v})\geq  \tilde{M}(\underline{u}+\underline v).
\end{aligned}
\end{equation*}
This proves the desired inequality for $\underline v$ with $M=\tilde M$ for the case $0\leq |x|<\psi_2(t)$.

Now we prove these inequalities for the case $\psi_2(t)\leq |x| \leq \underline{h}(t)$, where with
\[
s:=\min\Big\{1,\Psi_2(x,t)\lf(\frac{\underline{h}(t)-|x|}{(t+\sigma_2)^{\beta}}\rr)^{\lambda-1}\Big\}
\]
we have $\underline u=s\delta_1,\ \underline v\leq s\delta_2$, and therefore, by \eqref{linear1}, 
\begin{equation*}
-a_{22}\underline{v}+G(\underline{u})-\rho(\underline{u}+\underline{v})\geq -a_{22}s\delta_2+G(s\delta_1)-\rho(s\delta_1+s\delta_2)\geq 0.
\end{equation*}
Hence
\begin{equation*}\begin{aligned}
		 &d_2\int^{\underline h(t)}_{-\underline h(t)}J_2(x-y)\underline{v}(y,t)dy-d_2\underline{v}
	-a_{22}\underline{v}+G(\underline{u}) \\
	&\geq-d_2\epsilon \underline{v}+\rho(\underline{u}+\underline{v})\geq \frac{\rho}{2}(\underline{u}+\underline{v})\geq \tilde{M}(\underline{u}+\underline{v}),
\end{aligned}
\end{equation*}
for $\psi_2(t)\leq |x| \leq \underline{h}(t)$, $t>0$ with $\epsilon>0$ small and $\tilde M>0$ as before.
Moreover, for such $|x|$ we have
\[
\underline{u}(x,t)\leq \delta_1 \lf(\frac{\underline{h}(t)-\psi_2(t)}{(t+\sigma_2)^{\beta}}\rr)^{\lambda}\leq  \delta_1\lf(\frac{(C_1C_2+C_3)\ln(t+\sigma_2)}{(t+\sigma_2)^{\beta}}\rr)^{\lambda},
\]
and by \eqref{2-J-1bound},
\[
\int^{\underline{h}(t)}_{-\underline{h}(t)}J_1(x-y)\underline{u}(y,t)\rd y\geq \frac{2\tilde{K}_1\delta_1}{C_1(t+\sigma_2)\ln(t+\sigma_2)}.
\]
It follows that
\begin{equation*}\begin{aligned}
		& (d_1-\tilde{M})\int^{\underline h(t)}_{-\underline h(t)}J_1(x-y)\underline{u}(y,t)dy-(d_1+\tilde{M})\underline{u}
	-a_{11}\underline{u}+a_{12}\int^{\underline h(t)}_{-\underline h(t)}K(x-y)\underline{v}(y,t)dy\\
&\geq (d_1-\tilde{M})\int^{\underline h(t)}_{-\underline h(t)}J_1(x-y)\underline{u}(y,t)dy-(d_1+a_{11}+\tilde{M})\underline{u}\geq0\\
\end{aligned}
\end{equation*}
due to $\lambda\beta>1$  and $\sigma_2\gg1$. 

We have now proved all the  inequalities in \eqref{lower-inequa2} with $M=\tilde M$, and
hence \eqref{lowerJ1-1} holds. 
It remains to check the initial and  boundary conditions for $(\underline u, \underline v, -\underline h, \underline h)$ to be a lower solution, which are similar to the proof of Lemma \ref{alpha2}.
Since spreading happens to \eqref{1}, we have
\[\lim_{t\rightarrow \infty}(g(t),h(t))=(-\infty, \infty)\]
and 
\[
\lim_{t\rightarrow \infty}(u(x,t),v(x,t)) =(u^*,v^*) \mbox{ locally uniformly in }\R.
\]
Therefore, we can find a sufficiently large $T_0>0$ such that,
\[
[-\underline{h}(0),\underline{h}(0)]\subset[g(T_0),h(T_0)],
\]
and  for $x\in[-\underline{h}(0),\underline{h}(0)]$,
\[
u(x,T_0)\geq\delta_1\geq\underline{u}(x,0)\mbox{ and }v(x,T_0)\geq\delta_2>\underline{v}(x,0).
\]
Clearly $(\underline{u}(x,t),\underline{v}(x,t))=(0,0)$ for $x=\pm \underline{h}(t)$ and $t\geq0$. Therefore, we can apply the comparison principle to obtain that
\[
-\underline h(t)\geq g(T_0+t),
		 \ \ \underline{h}(t)\leq h(T_0+t) \mbox{ for }t\geq0,
\]
which clearly implies the desired inequality.
\end{proof}

\begin{lemma}\label{upper-J1dom}
Suppose that $(\mathbf{J})$ is satisfied. If spreading happens and \eqref{J_1dom} holds with $\alpha\in(1,2]$,
then there exists $ \tilde D:=\tilde D(\alpha)$ such that
\begin{equation*}
-g(t),\ h(t)\leq \begin{cases} \tilde D\,t^{\frac{1}{\alpha-1}} &\mbox{ if }\alpha\in(1,2);\\
\tilde D\, t\ln t &\mbox{ if }\alpha=2.
\end{cases}
\end{equation*}

\end{lemma}
\begin{proof}
We will construct suitable upper solutions to \eqref{1} for $\alpha\in(1,2)$ and $\alpha=2$, respectively.

We first consider the case  $\alpha\in(1,2)$ and define 
\begin{equation*}
\begin{aligned}
&\bar{h}(t)=(Ct+\sigma)^{\frac{1}{\alpha-1}},~~\bar{g}(t)=-\bar{h}(t),\\
&\bar{u}(x,t)=M u^*, \mbox{ for }x\in[-\bar{h}(t),\bar{h}(t)] \mbox{ and }t>0,\\
&\bar{v}(x,t)=M v^*,\mbox{ for }x\in[-\bar{h}(t),\bar{h}(t)]\mbox{ and }t>0,\\
\end{aligned}
\end{equation*}
where $C,\sigma>0$ are constants to be determined, and $M>1 $ satisfies $Mu^*\geq\|u_0\|_{\infty}$ and $Mv^*\geq\|v_0\|_{\infty}$.

We first prove the desired inequalities for $-\bar{h}'(t)$ and $\bar{h}'(t)$. By \eqref{J_1dom} and simple calculations we obtain
\begin{align*}
&\mu\int_{-\bar{h}(t)}^{\bar{h}(t)}\int^{+\infty}_{\bar{h}(t)}\Big(J_1(x-y)\bar{u}(x,t)+\rho J_2(x-y)\bar{v}(x,t)\Big) \rd y\rd x\\
\leq&\ \mu(\bar u+ \bar v)\int_{-2\bar{h}(t)}^{0}\int^{+\infty}_{0}[J_1(x-y)+\rho J_2(x-y)]\rd y\rd x\\
=&\ \mu(\bar u+ \bar v)\int^{2\bar{h}(t)}_{0}\int^{+\infty}_{x}[J_1(y)+\rho J_2(y)]\rd y\rd x\\
\leq &\ \mu
(1+\rho\tilde{K}_3)(\bar u+ \bar v)\int^{2\bar{h}(t)}_{0}\int^{+\infty}_{x}J_1(y)\rd y\rd x\\
= &\ \mu(1+\rho\tilde{K}_3)(\bar u+ \bar v)
\lf(\int^{2\bar{h}(t)}_{0}\int^{y}_0+\int^{+\infty}_{2\bar{h}(t)}\int^{2\bar{h}(t)}_0\rr)J_1(y)\rd x\rd y\\
=&\ \mu(1+\rho\tilde{K}_3)(\bar u+ \bar v)
\lf(\int^{2\bar{h}(t)}_{0}yJ_1(y)\rd y+\int^{+\infty}_{2\bar{h}(t)}2\bar{h}(t)J_1(y)\rd y\rr)\\
\leq&\ \mu
(1+\rho\tilde{K}_3)\tilde{K}_2(\bar u+ \bar v)
\lf(\int^1_0\rd y+\int^{2\bar{h}(t)}_{1}  y^{1-\alpha}\rd y+\int^{+\infty}_{2\bar{h}(t)}2\bar{h}(t)y^{-\alpha}\rd y \rr)\\
=&\ \mu(1+\rho\tilde{K}_3)\tilde{K}_2(\bar u+ \bar v)\lf(1+\frac{1}{2-\alpha}\lf[2^{2-\alpha}(Ct+\sigma)^{\frac{2-\alpha}{\alpha-1}}-1\rr]+\frac{2^{2-\alpha}}{\alpha-1}
(Ct+\sigma)^{\frac{2-\alpha}{\alpha-1}}\rr)\\
\leq &\ \frac C{\alpha-1}(Ct+\sigma)^{\frac{2-\alpha}{\alpha-1}}=\bar h'(t)
\end{align*}
provided that $C$ and $\sigma$ are sufficiently large.
Since $J_1(x)$ and $(\bar u, \bar v)$ are even in $x\in[-\bar h(t),\bar h(t)]$, the same argument gives that
\[
-\bar{h}'(t)\leq-\mu \int_{-\bar{h}(t)}^{\bar{h}(t)}\int_{-\infty}^{-\bar{h}(t)}\Big(J_1(x-y)\bar{u}(x,t)+\rho J_2(x-y)\bar{v}(x,t)\Big)dydx.
\]

Next we prove the desired inequalities for $\bar u$ and $\bar v$. By the definition  of $(\bar u,\bar v)$, clearly
\begin{align*}
&d_1\int^{\bar{h}(t)}_{-\bar{h}(t)}J_1(x-y)\bar u(y,t)dy-d_1\bar u-a_{11}\bar u+
a_{12}\int^{\bar{h}(t)}_{-\bar{h}(t)}K(x-y)\bar v(y,t)dy\\
& \leq d_1\bar u-d_1\bar u -a_{11}\bar u+a_{12}\bar v=0= \bar u_t,\\
&  d_2\int^{\bar{h}(t)}_{-\bar{h}(t)}J_2(x-y)\bar v(y,t)dy-d_2\bar v-a_{22}\bar v+G(\bar u)
\leq d_2\bar v-d_2\bar v -a_{22}\bar v+G (\bar u)\\
&=-a_{22}M v^*+G(Mu^*)\leq M[-a_{22}v^*+G(u^*)]=0= \bar v_t,
\end{align*}
where we have used the assumption $(G2)$.

Since $\sigma\gg1$ we have
\[
\bar{h}(0)=\sigma^{\frac{1}{\alpha-1}}\geq h_0 \mbox{ for }\alpha\in(1,2).
\]
By definition,
\[
\bar u(x,0)\geq u(x,0) \mbox{ and }\bar v(x,0)\geq v(x,0) \mbox{ for }x\in[-h_0,h_0],
\]
and
\[
\bar u(x,t)>0,~\bar v(x,t)> 0 \mbox{ for }x=\pm \bar h(t).
\]
Hence we can use the comparison principle to conclude that
\begin{equation*}
h(t)\leq \bar h(t) \mbox{ and }g(t)\geq -\bar h(t)\mbox{ for }t\geq0,
\end{equation*}
which yields the desired inequality for the case $\alpha \in (1,2)$.

For the case $\alpha=2$, we define
\begin{equation*}
\begin{aligned}
&\bar{h}(t):=(Ct+\sigma)\ln(Ct+\sigma),~~\bar{g}(t):=-\bar{h}(t), ~t>0\\
\end{aligned}
\end{equation*}
with positive constants $C$ and $\sigma$ to be determined, and $(\bar u,\bar v)$ the same as in the case $\alpha\in(1,2).$

It suffices to prove that the desired inequality  for $ \bar{h}(t)$, since the other desired inequalities can be proved in exactly the same way as for $\alpha\in(1,2).$
We have 
\begin{align*}
&\mu\int_{-\bar{h}(t)}^{\bar{h}(t)}\int^{+\infty}_{\bar{h}(t)}\Big(J_1(x-y)\bar{u}(x,t)+\rho J_2(x-y)\bar{v}(x,t)\Big) \rd y\rd x\\
\leq &\ \mu
(1+\rho\tilde{K}_3)(\bar u+ \bar v)\int^{2\bar{h}(t)}_{0}\int^{+\infty}_{x}J_1(y)\rd y\rd x\\
\leq&\ \mu(1+\rho\tilde{K}_3)(\bar u+ \bar v)
\lf(\int^{2\bar{h}(t)}_{0}yJ_1(y)\rd y+\int^{+\infty}_{2\bar{h}(t)}2\bar{h}(t)J_1(y)\rd y\rr)\\
\leq&\ \mu
(1+\rho\tilde{K}_3)\tilde{K}_2(\bar u+ \bar v)
\lf(\int^1_0\rd y+\int^{2\bar{h}(t)}_{1}  y^{-1}\rd y+\int^{+\infty}_{2\bar{h}(t)}2\bar{h}(t)y^{-2}\rd y \rr)\\
=&\ \mu(1+\rho\tilde{K}_3)\tilde{K}_2(\bar u+ \bar v)\lf(2+\ln [2(Ct+\sigma)\ln(Ct+\sigma)] \rr)\\
=&\ \mu(1+\rho\tilde{K}_3)\tilde{K}_2(\bar u+ \bar v)\lf(2+\ln 2 +\ln(Ct+\sigma)+\ln\ln(Ct+\sigma)\rr)\\
\leq &\ C\ln(Ct+\sigma)+C=\bar h'(t),
\end{align*}
if $C$ and $\sigma$ are sufficiently large. We may then apply the comparison principle to obtain the desired result.
\end{proof}

\end{document}